\documentclass{amsart}

\usepackage{amsthm}
\usepackage{amssymb}

\title{The open and clopen Ramsey theorems in the Weihrauch lattice}

\author{Alberto Marcone}
\address{Department of Mathematics, Computer Science and Physics\newline
University of Udine\newline
Udine, UD 33100, IT}
\email{alberto.marcone@uniud.it}

\author{Manlio Valenti}
\address{Department of Mathematics, Computer Science and Physics\newline
University of Udine\newline
Udine, UD 33100, IT}
\email{manlio.valenti@uniud.it}

\thanks{We thank Arno Pauly, Vasco Brattka, Jun Le Goh and Eleonora Pippia for useful discussions about the topic of the paper.\\
We also thank the anonymous referee for his/her careful reading of the paper and many valuable suggestions that improved the presentation. In particular (s)he noticed that the proof of \thref{thm:lim^n*openRT<=findHopen} yielded a stronger result than the one we originally stated. \\
Both author's research was partially supported by the departmental PRID 
funding \lq\lq HiWei --- The higher levels of the Weihrauch 
hierarchy\rq\rq\ and by the Italian PRIN 2017 Grant \lq\lq Mathematical 
Logic: models, sets, computability\rq\rq.}

\makeatletter
\@namedef{subjclassname@2020}{\textup{2020} Mathematics Subject Classification}
\makeatother



\usepackage{stmaryrd}
\SetSymbolFont{stmry}{bold}{U}{stmry}{m}{n}

\usepackage{tikzit}
\tikzstyle{box}=[fill={rgb,255: red,228; green,228; blue,228}, draw=black, shape=rectangle]

\tikzstyle{reducible}=[->, dashed]
\tikzstyle{strictreducible}=[->]
\tikzstyle{nonreducible}=[->, draw=red]
\tikzstyle{uncomp}=[draw=red, <->]

\tikzstyle{thmref}=[opacity=0,inner sep=2mm]% text width=3mm, text height=3mm]

\usepackage{ifthen}

\newcommand{\charfun}[1]{\chi_{{#1}}}
\newcommand{\setdifference}{\backslash}
\newcommand{\setcomplement}[2][]{\ifthenelse{\equal{#1}{}}{#2^\mathrm{C}}{ {#1}\setdifference {#2} }}
\newcommand{\id}[1]{\operatorname{id}_{ {#1} }}
\newcommand{\idBaire}{ \id{\Baire} } 

\newcommand{\restrict}[1]{\ensuremath{\left. \hspace{-1mm} \right|_{#1}}}

\newcommand{\wjump}[1]{{#1}'}

\newcommand{\textdef}[1]{\textit{#1}}

\newcommand{\abslength}[1]{|#1|}
\newcommand{\prefix}{\sqsubseteq}
\newcommand{\pprefix}{\sqsubset}
\newcommand{\coding}[1]{\langle #1 \rangle}
\newcommand{\concat}{\raisebox{.9ex}{\ensuremath\smallfrown} }

\newcommand{\substring}{\preceq}
\newcommand{\finiteSubstring}{\preceq^*}
\newcommand{\dominated}{\trianglelefteq}

\newcommand{\finStrings}[1]{{#1}^{<\mathbb{N}}}
\newcommand{\infStrings}[1]{{#1}^{\mathbb{N}}}

\newcommand{\Baire}{{\mathbb{N}^\mathbb{N}}}
\newcommand{\finBaire}{{\mathbb{N}^{<\mathbb{N}}}}

\newcommand{\Ramsey}{{[\mathbb{N}]}^\mathbb{N} }
\newcommand{\incstring}{{[\mathbb{N}]}^{<\mathbb{N}} }

\newcommand{\Cantor}{{2^\mathbb{N}}}

\newcommand{\pfunction}{:\subseteq}
\newcommand{\mfunction}[2]{: #1 \rightrightarrows #2}
\newcommand{\pmfunction}[2]{\pfunction #1 \rightrightarrows #2}

\newcommand{\strictlyturingreducible}{<_T}

\newcommand{\weireducible}{\le_{\mathrm{W}}}
\newcommand{\strictlyweireducible}{<_\mathrm{W}}
\newcommand{\weiequiv}{\equiv_{\mathrm{W}}}
\newcommand{\weiincomparable}{~|_{\mathrm{W}~}}
\newcommand{\strongweireducible}{\le_{\mathrm{sW}}}

\newcommand{\strongweiequiv}{\equiv_{\mathrm{sW}}}

\newcommand{\weiarithreducible}{\le^a_{\mathrm{W}}}
\newcommand{\strictlyweiarithreducible}{<^a_\mathrm{W}}
\newcommand{\weiarithequiv}{\equiv^a_{\mathrm{W}}}

\newcommand{\KleeneO}{{\mathcal{O}}}

\newcommand{\boldfaceDelta}{\boldsymbol{\Delta}}
\newcommand{\boldfaceSigma}{\boldsymbol{\Sigma}}
\newcommand{\boldfacePi}{\boldsymbol{\Pi}}
\newcommand{\boldfaceGamma}{\boldsymbol{\Gamma}}

\newcommand{\repmap}[1]{\delta_{#1}}
\newcommand{\realizer}{\vdash}

\newcommand{\parallelization}[1]{\widehat{#1}}

\newcommand{\compproduct}{*}

\newcommand{\LO}{\mathrm{LO}}
\newcommand{\WO}{\mathrm{WO}}

\newcommand{\repTree}{\mathbf{Tr}}
\newcommand{\repincTree}{\mathbf{Ti}}

\newcommand{\KB}{\mathrm{KB}}

\newcommand{\ran}{\operatorname{ran}}
\newcommand{\dom}{\operatorname{dom}}

\newcommand{\st}{:}
\newcommand{\sequence}[2]{(#1)_{#2}}

\newcommand{\ball}[2]{{B(#1,#2)}}

\newcommand{\HomSol}{\mathrm{HS}}
\newcommand{\closedsidesoltree}[1]{T_{{#1}}}
\newcommand{\stringcodepower}[1]{\psi_{{#1}}}
\newcommand{\elementpower}[1]{\eta_{{#1}}}
\newcommand{\opencodepairing}{\boxtimes}
\newcommand{\solovay}[1]{W_{#1}}
\newcommand{\describepath}[1]{D_{{#1}}}

\newcommand{\openRamsey}{\boldfaceSigma^0_1\mathsf{-RT}}
\newcommand{\clopenRamsey}{\boldfaceDelta^0_1\mathsf{-RT}}
\newcommand{\findHopen}{\mathsf{FindHS}_{\boldfaceSigma^0_1} }
\newcommand{\wfindHopen}{\mathsf{wFindHS}_{\boldfaceSigma^0_1} }
\newcommand{\TwfindHopen}{\mathsf{TwFindHS}_{\boldfaceSigma^0_1} }
\newcommand{\findHclosed}{\mathsf{FindHS}_{\boldfacePi^0_1}}
\newcommand{\wfindHclosed}{\mathsf{wFindHS}_{\boldfacePi^0_1}}

\newcommand{\findHclopen}{\mathsf{FindHS}_{\boldfaceDelta^0_1}}

\newcommand{\wfindHclopen}{\mathsf{wFindHS}_{\boldfaceDelta^0_1}}

\newcommand{\atrformula}{\mathrm{H}}
\newcommand{\ATR}{{ \mathsf{ATR} }}

\newcommand{\LPO}{ \mathsf{LPO} }
\newcommand{\NHA}{ \mathsf{NHA} }

\newcommand{\codedChoice}[3]{\ifthenelse{\equal{#1}{}}{\mathsf{C}^{\mathsf{#2}}_{{#3}} }{ {#1}\text{-}\mathsf{C}^{\vphantom{g}\mathsf{#2}}_{{#3}} }}
\newcommand{\codedTChoice}[3]{\ifthenelse{\equal{#1}{}}{\mathsf{TC}^{\mathsf{#2}}_{{#3}} }{ {#1}\text{-}\mathsf{TC}^{\vphantom{g}\mathsf{#2}}_{{#3}} }}

\newcommand{\Choice}[1]{\codedChoice{}{}{#1}}
\newcommand{\TChoice}[1]{\codedTChoice{}{}{#1}}

\newcommand{\codedUChoice}[3]{\ifthenelse{\equal{#1}{}}{\mathsf{UC}^{\mathsf{#2}}_{{#3}} }{ {#1}\text{-}\mathsf{UC}^{\vphantom{g}\mathsf{#2}}_{{#3}} }}

\newcommand{\sigmaCofChoice}{\codedChoice{\boldfaceSigma^1_1}{cof}{\mathbb{N}}}
\newcommand{\pSigmaCofChoice}{\parallelization{\codedChoice{\boldfaceSigma^1_1}{cof}{\mathbb{N}}}}

\newcommand{\UCBaire}{\mathsf{UC}_\Baire}
\newcommand{\CBaire}{\mathsf{C}_\Baire}
\newcommand{\CCantor}{\mathsf{C}_\Cantor}

\newcommand{\TCBaire}{\mathsf{TC}_\Baire}
\newcommand{\sTCBaire}{\mathsf{sTC}_{\Baire}}
\newcommand{\sTCCantor}{\mathsf{sTC}_{\Cantor}}

\newcommand{\mflim}{\mathsf{lim}}
\newcommand{\SigmaSep}{\mathbf{\boldfaceSigma^1_1\mathsf{-Sep}}}

\newcommand{\chiPi}{\chi_{\Pi^1_1}}
\newcommand{\cchiPi}{c\chi_{\Pi^1_1}}

\input{thref}

\newtheorem{theorem}{Theorem}[section]
\newtheorem{proposition}[theorem]{Proposition}
\newtheorem{lemma}[theorem]{Lemma}
\newtheorem{corollary}[theorem]{Corollary}
\newtheorem{unprovedtheorem}[theorem]{Theorem}
\newtheorem{unprovedproposition}[theorem]{Proposition}
\theoremstyle{definition}
\newtheorem{definition}[theorem]{Definition}
\newtheorem{question}[theorem]{Question}

\usepackage[hidelinks]{hyperref}

\begin{document}
\begin{abstract}
	We investigate the uniform computational content of the open and clopen Ramsey theorems in the Weihrauch lattice. While they are known to be equivalent to $\mathrm{ATR_0}$ from the point of view of reverse mathematics, there is not a canonical way to phrase them as multivalued functions. We identify $8$ different multivalued functions ($5$ corresponding to the open Ramsey theorem and $3$ corresponding to the clopen Ramsey theorem) and study their degree from the point of view of Weihrauch, strong Weihrauch and arithmetic Weihrauch reducibility. In particular one of our functions turns out to be strictly stronger than any previously studied multivalued functions arising from statements around $\mathrm{ATR}_0$.
\end{abstract}
\maketitle

\tableofcontents

\section{Introduction}
This work explores the uniform computational strength of some infinite-dimensional generalizations of Ramsey's theorem. The classical finite-dimensional Ramsey theorem can be stated as follows: for every $A\subset \mathbb{N}$ let $[A]^n:=\{ B\subset A\st |B|=n\}$ be the set of subsets of $A$ with cardinality $n$. A map $c\colon[\mathbb{N}]^n\to k$ is called a \textdef{$k$-coloring} of $[\mathbb{N}]^n$. A set $H$ s.t.\ $c([H]^n)=\{i\}$ for some $i<k$ is called \textdef{homogeneous} for $c$. 
\begin{theorem}[Ramsey's theorem]
	For every $n,k\ge 1$ and every coloring $c\colon[\mathbb{N}]^n\to k$ there is an infinite subset $H\subset \mathbb{N}$ that is homogeneous for $c$.
\end{theorem}
There is a vast literature exploring the computational aspects of Ramsey's theorem. Our focus will be on the infinite generalization of the above result, i.e.\ we will consider $n=\infty$. In particular we will focus on Nash-Williams' theorem, also called the open Ramsey theorem: 
\begin{theorem}[Nash-Williams \cite{NashWilliams65}]
	The open subsets of $\Ramsey$ admit infinite homogeneous sets.
\end{theorem}
We will also consider the restriction of Nash-Williams' theorem to clopen subsets of $\Ramsey$. 

The notion of Weihrauch reducibility provides a useful tool to explore the uniform computational content of theorems: indeed many statements in mathematics can be written in the form $(\forall x\in X)(\varphi(x)\rightarrow(\exists y\in Y)(\psi(x,y)))$, which can be thought of as a computational problem where $x$ is an instance of a problem and the goal is to find a solution $y$ (which is, in general, not unique). This can be naturally written in the language of partial multivalued functions by considering $f\pmfunction{X}{Y}$ s.t.\ $f(x)=\{y\in Y \st \psi(x,y)\}$, for every $x\in X$ s.t.\ $\varphi(x)$.

Starting with the work of Gherardi and Marcone \cite{GM09}, it has been shown that Weihrauch reducibility provides a bridge between computable analysis and reverse mathematics. In particular reverse mathematics focuses on proving the equivalence of statements over a weak base theory. Historically, a large number of theorems turned to be equivalent to one of the so-called \textdef{big five}: $\mathrm{RCA}_0$, $\mathrm{WKL}_0$, $\mathrm{ACA}_0$, $\mathrm{ATR}_0$, $\boldfacePi^1_1\mathrm{-CA}_0$. Oftentimes, the close connection between reverse mathematics and Weihrauch reducibility has been exploited to translate results from one setting into the other.

Following this connection, there are a number of established ``analogues'' of the big-five in the Weihrauch lattice: $\mathrm{RCA}_0$ corresponds to computable problems, $\mathrm{WKL}_0$ corresponds to $\CCantor$ (i.e.\ choosing an element from a non-empty closed subset of $\Cantor$) and $\mathrm{ACA}_0$ corresponds to iterations of $\mflim$. 

Recently Marcone \cite{Dagstuhl2016} raised the question ``What do the Weihrauch hierarchies look like once we go to very high levels of reverse mathematics strength?". There have been several works in this direction: Kihara, Marcone and Pauly \cite{KMP20} have studied several principles, like the (strong) comparability of well-orders, the perfect tree theorem and the open determinacy theorem; Goh \cite{gohatr,GohThesis} analyzed the weak comparability of well-orders and the K\"onig duality theorem; Angl\`es D'Auriac and Kihara \cite{KiharaADauriacChoice} dealt with the $\Sigma^1_1$ choice on $\mathbb{N}$ and variants thereof.

Our work explores the formalization of the open and clopen Ramsey theorems as multivalued functions and analyzes their position in the Weihrauch lattice (both the open and the clopen Ramsey theorems are known to be equivalent to $\mathrm{ATR}_0$ over $\mathrm{RCA}_0$, see \cite[Sec.\ V.9]{Simpson09}). Notice that, as already occurred to other principles equivalent to $\mathrm{ATR}_0$ (\cite{KMP20,GohThesis}), there is not a single multivalued function corresponding to the open Ramsey theorem. Actually, in our case, the situation is even more complex than for the open determinacy or the perfect tree theorem, as the two alternatives (homogeneous solution on the open side or homogeneous solution on the closed side) given by the open Ramsey theorem are not mutually exclusive. Therefore given an open set we can ask for a homogeneous solution on the open side, a homogeneous solution on the closed side or a homogeneous solution on either side. Altogether we will define five different multivalued functions corresponding to the open Ramsey theorem and three different functions corresponding to the clopen Ramsey theorem.

While it is hard to point out a single analogue of $\mathrm{ATR}_0$ in the Weihrauch lattice, there are a number of degrees that play a central role in calibrating the strength of principles that, from the point of view of reverse mathematics, lie at the level of $\mathrm{ATR}_0$. Prominent among these are $\UCBaire$, $\CBaire$ and $\TCBaire$ (see Section \ref{subsec:wei} for their precise definitions).
It is known that $\UCBaire\strictlyweireducible\CBaire\strictlyweireducible\TCBaire$ (see \cite{KMP20}).

When dealing with multivalued functions that are very high in the Weihrauch lattice it is often appropriate to use arithmetic Weihrauch reducibility rather than (computable) Weihrauch reducibility. 

In Figure \ref{fig:summary} we summarize the results we obtain both with respect to Weihrauch reducibility and arithmetic Weihrauch reducibility. Notice that the multivalued function $\findHopen$ is stronger than any multivalued function related to $\mathrm{ATR}_0$ considered so far. In fact all these functions are strictly Weihrauch reducible to $\TCBaire^*$, which, by \thref{thm:tcbaire^*<findHopen}, is strictly below $\findHopen$. Notice also that, since $\findHopen$ is closed under parallel product (\thref{prop:findHopen_closed_product}), it computes $\sTCBaire^* \weiequiv \TCBaire^*\times \chiPi^*$, which was suggested as an $\mathrm{ATR}_0$ analogue in \cite[Sec.\ 9]{KMP20}.
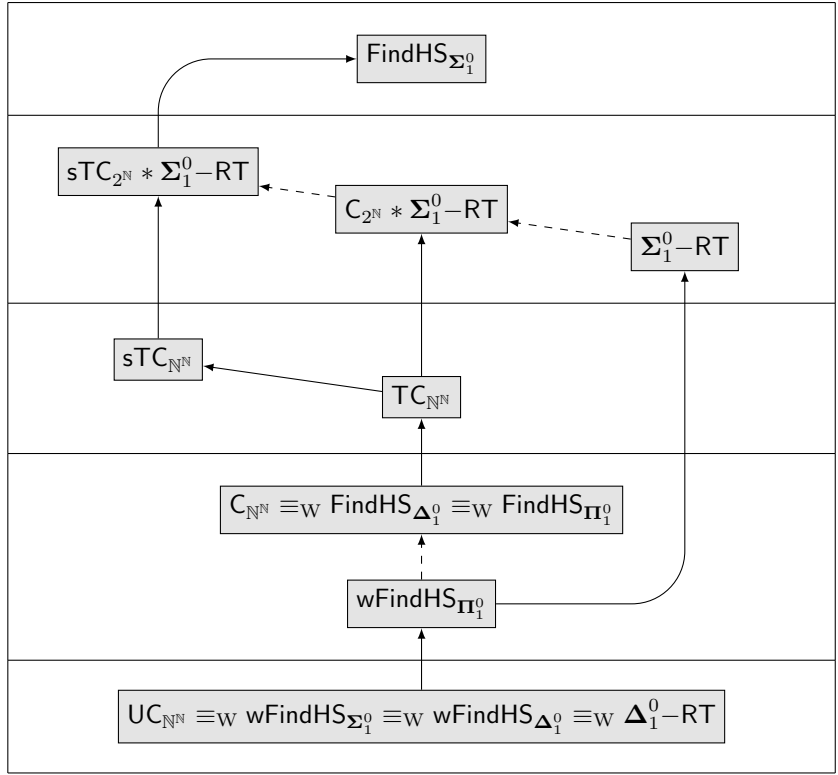
\begin{figure}
\begin{center}
	\tikzstyle{every picture}=[tikzfig]
	\begin{tikzpicture}
%	\begin{pgfonlayer}{nodelayer}
		\node [style=box] (UCBaire) at (0, -6.5) {$\UCBaire\weiequiv \wfindHopen \weiequiv \wfindHclopen \weiequiv \clopenRamsey $};
		\node [style=box] (wfindHclosed) at (0, -3.5) {$\wfindHclosed$};
		\node [style=box] (CBaire) at (0, -1) {$\CBaire \weiequiv \findHclopen \weiequiv \findHclosed$};
		\node [style=box] (TCBaire) at (0, 2) {$\TCBaire$};
		\node [style=box] (sTCBaire) at (-7, 3) {$\sTCBaire$};
		\node [style=box] (openRamsey) at (7, 6) {$\openRamsey$};
		\node [style=box] (CCantor*openRamsey) at (0, 7) {$\CCantor\compproduct\openRamsey$};
		\node [style=box] (sTCCantor*openRamsey) at (-7, 8) {$\mathsf{sT}\CCantor\compproduct\openRamsey$};
		\node [style=box] (findHopen) at (0, 11) {$\findHopen$};
%		\node [opacity=0] at (wfindHclosed.south) {\ref{prop:ucbaire<wfindHclosed}};
		\node [style=none] (35) at (-11, -8) {};
		\node [style=none] (36) at (-11, -5) {};
		\node [style=none] (37) at (11, -5) {};
		\node [style=none] (38) at (11, -8) {};
		\node [style=none] (39) at (-11, 0.5) {};
		\node [style=none] (40) at (11, 0.5) {};
		\node [style=none] (41) at (-11, 4.5) {};
		\node [style=none] (42) at (11, 4.5) {};
		\node [style=none] (43) at (-11, 9.5) {};
		\node [style=none] (44) at (-11, 12.5) {};
		\node [style=none] (45) at (11, 9.5) {};
		\node [style=none] (46) at (11, 12.5) {};
%	\end{pgfonlayer}
%	\begin{pgfonlayer}{edgelayer}
		\drawThm{ [style=strictreducible] (UCBaire) -- (wfindHclosed) }{}{\ref{prop:ucbaire<wfindHclosed}};
		\drawThm{ [style=reducible] (wfindHclosed) -- (CBaire) }{}{\ref{prop:wfindHclosed<=cbaire}};
		\drawThm{ [style=strictreducible, rounded corners=20] (wfindHclosed.east) -| (openRamsey)}{}{\ref{cor:wfindHclosed<openRT}};
		\drawThm{ [style=strictreducible] (CBaire) -- (TCBaire)}{}{};
		\drawThm{ [style=strictreducible] (TCBaire) -- (CCantor*openRamsey)}{}{\ref{prop:tcbaire<=ccantor_comp_openRamsey}};
		\drawThm{ [style=strictreducible] (TCBaire) -- (sTCBaire)}{}{\ref{thm:tcbaire<stcbaire}};
		\drawThm{ [style=strictreducible] (sTCBaire) -- (sTCCantor*openRamsey)}{}{\ref{thm:sTCBaire<=sTCCantor*openRT}};
		\draw [style=reducible] (openRamsey) -- (CCantor*openRamsey);
		\draw [style=reducible] (CCantor*openRamsey) -- (sTCCantor*openRamsey);
		\drawThm{ [style=strictreducible, rounded corners=20] (sTCCantor*openRamsey) |- (findHopen)}{}{\ref{thm:lim^n*openRT<=findHopen}};
		\draw (44.center) to (43.center);
		\draw (43.center) to (41.center);
		\draw (41.center) to (39.center);
		\draw (39.center) to (36.center);
		\draw (36.center) to (35.center);
		\draw (35.center) to (38.center);
		\draw (38.center) to (37.center);
		\draw (37.center) to (40.center);
		\draw (40.center) to (42.center);
		\draw (44.center) to (46.center);
		\draw (46.center) to (45.center);
		\draw (45.center) to (42.center);
		\draw (43.center) to (45.center);
		\draw (41.center) to (42.center);
		\draw (39.center) to (40.center);
		\draw (36.center) to (37.center);
%	\end{pgfonlayer}
\end{tikzpicture}
	\caption{Multivalued functions related to the open and clopen Ramsey theorems in the Weihrauch lattice. Dashed arrows represent Weihrauch reducibility in the direction of the arrow, solid arrows represent strict Weihrauch reducibility. The large rectangles indicate arithmetic Weihrauch equivalence classes. In particular, every function strictly arithmetically reduces to all the functions in rectangles above its own.}
	\label{fig:summary}
\end{center}
\end{figure}

\subsection{Structure of the paper}
In Section \ref{sec:background} we briefly introduce the notation and recall the 
preliminary notions on represented spaces (Section \ref{subsec:rep_spaces}) and Weihrauch reducibility (Section \ref{subsec:wei}). In Section \ref{sec:ramsey} we will recall the precise statement for the open and clopen Ramsey theorems and prove some lemmas that will be useful in proving the results on the Weihrauch degrees. The reader may skip these lemmas on the first read, and return to it as needed. In Section \ref{sec:rt_wei} we define the multivalued functions corresponding to the open and clopen Ramsey theorems and study their degrees. In particular we divide the analysis into: functions that are reducible to $\UCBaire$ (Section \ref{subsec:ucbaire}), functions that are reducible to $\CBaire$ (but not to $\UCBaire$, Section \ref{subsec:cbaire}) and functions that are not reducible to $\CBaire$ (Section \ref{subsec:not_cbaire}). Moreover, in Section \ref{subsec:strong_wei} we characterize the strength of these functions from the point of view of strong Weihrauch reducibility. Finally, in Section \ref{sec:artihemtic_reducibility} we focus on the behavior of these functions under arithmetic Weihrauch reducibility and in Section \ref{sec:conclusions} we draw some conclusions and list some open problems. 

\section{Background}
\thlabel{sec:background}
	We refer the reader to \cite{Simpson09} for a detailed introduction to reverse mathematics. 

	Here we introduce the notation we will use in the rest of the paper. 

	Let $\Baire$ be the Baire space endowed with the usual product topology and let $\finBaire$ be the set of finite sequences of natural numbers. Let also $\Ramsey$ be the space of total functions $f\colon \mathbb{N}\to\mathbb{N}$ that are strictly increasing (i.e.\ $n<m\Rightarrow f(n)<f(m)$). This is sometimes called \textdef{Ramsey space} (see e.g.\ \cite[Sec.\ V.9]{Simpson09}). We will use the symbol $\prefix$ for the prefix relation. For $\sigma,\tau\in\finBaire$ we write: $\abslength{\sigma}$ for the length of $\sigma$, $\sigma\concat \tau$ for the concatenation of $\sigma$ and $\tau$, and $\sigma\dominated \tau$ for the domination relation, i.e.\ $\abslength{\sigma}=\abslength{\tau}$ and $(\forall i<\abslength{\sigma})(\sigma(i)\le \tau(i))$. The domination relation applies also to infinite strings. Also if $x\in\finBaire\cup \Baire$ we write $x[m]$ for the prefix of $x$ of length $m$. If $f,g\in\Baire$ we denote the composition $f\circ g$ by $fg$. Moreover, we write $f\substring g$ if $f$ is a subsequence of $g$, i.e.\ if there exists $h\in\Ramsey$ s.t.\ $f=gh$. In particular for every $h\in\Ramsey$ we have $gh\substring g$. Similarly, we write $\sigma\finiteSubstring f$ if $\sigma$ is a finite subsequence of $f$.

\subsection{Represented spaces}
\label{subsec:rep_spaces}
	We briefly introduce the main concepts on the theory of represented spaces. For a more thorough presentation we refer to \cite{Pauly16,Weihrauch00}.
	
	A \textdef{represented space} is a pair $(X,\repmap{X})$ where $X$ is a set and $\repmap{X}\pfunction \Baire \to X$ is a partial surjection. For every $x\in X$ a \textdef{name} or \textdef{code} for $x$ is any element of $\repmap{X}^{-1}(x)$. When it is clear from the context, we omit to explicitly mention the representation map $\repmap{X}$.

	Given two representation maps $\delta$ and $\delta'$ of the same set $X$, we say that $\delta$ is \textdef{reducible} to $\delta'$ (and we write $\delta\le_c \delta'$) if there is a computable map $F\pfunction\Baire\to\Baire$ s.t.\ $\delta(p)=\delta'(F(p))$ for every $p\in \dom(\delta)$. The maps $\delta$ and $\delta'$ are called \textdef{equivalent} if $\delta\le_c\delta'$ and $\delta'\le_c\delta$. 

	Let $(X,\repmap{X})$ and $(Y,\repmap{Y})$ be represented spaces and $f\pmfunction{X}{Y}$ be a partial multivalued function. We say that a partial function $F\pfunction \Baire\to\Baire$ is a \textdef{realizer}\footnote{The very existence of realizers for every $f$ depends on some form of the axiom of choice.} for $f$ (and we write $F\realizer f$) if, for every $p\in\dom(f\circ \repmap{X})$ we have that $\repmap{Y}(F(p))\in f(\repmap{X}(p))$. The notion of realizer allows us to transfer properties of functions on the Baire space (such as computability or continuity) to multivalued functions on represented spaces. In particular we say that a multivalued function between represented spaces is computable if it has a computable realizer. 

	Let $(X,d)$ be a separable metric space and $\alpha\colon\mathbb{N}\to X$ be a dense sequence in $X$. The tuple $(X,d,\alpha)$ is called a \textdef{computable metric space} if $d\circ(\alpha\times\alpha)\colon\mathbb{N}^2\to\mathbb{R}$ is a computable double sequence in $\mathbb{R}$. 

	Fix a computable enumeration $\sequence{q_n}{n\in\mathbb{N}}$ of $\mathbb{Q}$. If $(X,d,\alpha)$ is a separable metric space, for every $k\ge 1$ we can define the represented spaces $(\boldfaceSigma^0_k(X),\repmap{\boldfaceSigma^0_k(X)})$, $(\boldfacePi^0_k(X),\repmap{\boldfacePi^0_k(X)})$, $(\boldfaceDelta^0_k(X),\repmap{\boldfaceDelta^0_k(X)})$ inductively by: 
	\begin{itemize}
		\item $\repmap{\boldfaceSigma^0_1(X)}(p):=\bigcup_{\coding{i,j}\in\ran(p)}\ball{\alpha(i)}{q_j}$;
		\item $\repmap{\boldfacePi^0_k(X)}(p):=\setcomplement[X]{\repmap{\boldfaceSigma^0_k(X)}(p)}$;
		\item $\repmap{\boldfaceSigma^0_{k+1}(X)}(\coding{p_0,p_1,\hdots}):=\bigcup_{i\in\mathbb{N}} \repmap{\boldfacePi^0_k(X)}(p_i)$;
		\item $\repmap{\boldfaceDelta^0_k(X)}(\coding{p,q}):=\repmap{\boldfaceSigma^0_k(X)}(p)$, iff $p\in\dom(\repmap{\boldfaceSigma^0_k(X)})$, $q\in\dom(\repmap{\boldfacePi^0_k(X)})$ and $\repmap{\boldfaceSigma^0_k(X)}(p)=\setcomplement[X]{\repmap{\boldfacePi^0_k(X)}(q)}$.
	\end{itemize}
	The represented spaces $(\boldfaceSigma^0_k(X),\repmap{\boldfaceSigma^0_k(X)})$, $(\boldfacePi^0_k(X),\repmap{\boldfacePi^0_k(X)})$ can be defined for any represented space $(X,\repmap{X})$, using the fact that the category of represented spaces is cartesian-closed and letting a name for an open set be the name of its characteristic function, where the codomain $\{0,1\}$ is equipped with the Sierpi\'nski topology.	For separable metric spaces, the two representations are equivalent (see \cite{Pauly16,Brattka05}).

	\begin{lemma}
		\thlabel{lem:computable_operations_open}
		The following maps are computable:
		\begin{enumerate}
			\item \label{itm:includsion_clopen_open} $\boldfaceDelta^0_1(\Ramsey)\hookrightarrow \boldfaceSigma^0_1(\Ramsey):=D\mapsto D$;
			\item \label{itm:compl_clopen_set} $\boldfaceDelta^0_1(\Ramsey)\to \boldfaceDelta^0_1(\Ramsey):=D\mapsto \setcomplement[\Ramsey]{D}$;
			\item \label{itm:open_union} $\cup\colon\boldfaceSigma^0_1(\Ramsey)\times\boldfaceSigma^0_1(\Ramsey)\to\boldfaceSigma^0_1(\Ramsey):=(P,Q)\mapsto P \cup Q$.
		\end{enumerate}
	\end{lemma}
	\begin{proof}
		${}$
		\begin{description}
			\item[\ref{itm:includsion_clopen_open}, \ref{itm:compl_clopen_set}] follow from the fact that a name for a clopen set is the join $\coding{p,q}$ of two names for open sets (one for the set and one for its complement);
			\item[\ref{itm:open_union}] see \cite[Prop. 3.2(5)]{Brattka05}.\qedhere
		\end{description}		
	\end{proof}

	The set $\boldfaceSigma^1_1(X)$ of analytic subsets of $X$ can be seen as a represented space by defining a name for $S$ to be a name for a closed set $A\subset X\times \Baire$ s.t.\ $S=\pi_X (A)$ (where $\pi$ denotes projection). Moreover, we can define a name for a coanalytic set $R\in\boldfacePi^1_1(X)$ to be a name for its complement.

	We denote with $\repTree$ the space of trees on $\mathbb{N}$ represented via their characteristic function. Similarly we denote with $\repincTree$ the space of trees with strictly increasing strings, represented analogously. 
	The function $[\cdot]\colon \repTree\to\boldfacePi^0_1(\Baire)$ that maps a tree to the set of its paths is computable with multivalued computable inverse. This implies that a closed set $A$ of $\Baire$ or $\Ramsey$ can be equivalently represented via the characteristic function of a tree $T$ s.t.\ $[T]=A$.
	
	Similarly, an open set $P$ of $\Baire$ can be equivalently represented via an enumeration $p$ of a prefix-free subset of $\finBaire$ s.t.\ $P=\{ f\in\Baire \st (\exists i)(p(i) \pprefix f) \}$. With a small abuse of notation we write $\tau\in p$ in place of $\tau \in \ran(p)$. The same considerations can be made for the space $\Ramsey$.

	We denote by $\LO=(\LO,\repmap{\LO})$ the represented space of countable linear orders, where an order $L$ is represented by the characteristic function of the relation $\{\coding{a,b}\st a\le_L b\}$. We also denote by $\WO=(\WO,\repmap{\WO})$ the represented space of countable well-orders, where the representation map $\repmap{\WO}$ is the restriction of $\repmap{\LO}$ to codes of well-orders.

	For every tree $T\subset\finBaire$ we denote by $\KB(T)$ the Kleene-Brouwer order on $T$, defined as $\sigma \le_{\KB(T)} \tau$ iff $\sigma,\tau \in T$ and $\tau \prefix\sigma$ or $\sigma\le_{lex} \tau$. The map $T\mapsto \KB(T)$ from $\repTree$ to $\LO$ is computable. It is known that $\KB(T)$ is a well-order iff $[T]=\emptyset$ (see e.g.\ \cite[Lem.\ V.1.3]{Simpson09}).

\subsection{Weihrauch reducibility}
\label{subsec:wei}
	Let $f\pmfunction{X}{Y}$ and $g\pmfunction{Z}{W}$ be multivalued functions between represented spaces. We say that $f$ is \textdef{Weihrauch reducible} to $g$ and we write $f\weireducible g$ iff there are two computable maps $\Phi,\Psi\pfunction \Baire\to\Baire$ s.t., for every realizer $G\realizer g$ we have $p\mapsto \Psi(\coding{ p,G\Phi(p) })$ is a realizer for $f$. We say that $f$ is \textdef{strongly Weihrauch reducible} to $g$, and we write $f\strongweireducible g$, if there are two computable maps $\Phi,\Psi\pfunction \Baire\to\Baire$ s.t., for every realizer $G\realizer g$ we have $\Psi G\Phi \realizer f$. We often say that $\Phi$ is the \textdef{forward} functional of the (strong) Weihrauch reduction, while $\Psi$ is the \textdef{backward} functional.

	The relations $\weireducible$ and $\strongweireducible$ are reflexive and transitive, and therefore induce two degree structures on the family of multivalued functions on represented spaces. For more details on the algebraic properties of the Weihrauch and strong Weihrauch degrees we refer the reader to \cite{HiguchiPauly13,BP16,BGP17}. 

	In the following we briefly introduce the operations on multivalued functions that we will need in this work:
	\begin{description}
		\item [parallel product] $f\times g\pmfunction{X\times Z}{Y\times W}:=(x,z)\mapsto f(x)\times g(z)$ with $\dom(f\times g):=\dom(f)\times \dom(g)$;
		\item [finite parallelization] $f^*\pmfunction{\finStrings{X}}{\finStrings{Y}}$ with domain $\dom(f^*):=\{\sequence{x_i}{i<n}\st (\forall i<n)(x_i \in \dom(f)) \}$ defined as $f^*(\sequence{x_i}{i<n}):=\{ \sequence{y_i}{i<n}\st (\forall i<n)(y_i\in f(x_i)) \}$;
		\item [parallelization] $\parallelization{f}\pmfunction{\infStrings{X}}{\infStrings{Y}}$ with domain $\dom(\parallelization{f}):=\dom(f)^\mathbb{N}$ defined as $\parallelization{f}(\sequence{x_n}{n\in\mathbb{N}}):= \sequence{f(x_n)}{n\in\mathbb{N}}$;
		\item [jump] we define the \textdef{jump} of the represented space $(X,\repmap{X})$ as the represented space $X'=(X,\repmap{X'})$, where a $\repmap{X'}$-name for $x$ is a string $\coding{p_0,p_1,\hdots}$ s.t.\ $\repmap{X}(\lim_{n\to\infty} p_n)=x$. The jump of $f$ is defined as $\wjump{f}\pmfunction{X'}{Y}:=x\mapsto f(x)$. We write $f^{(n)}$ to denote the result of applying the jump operation $n$ times.
 	\end{description}
	We also define the \textdef{compositional product} as 
	\[ f*g:=\max_{\weireducible}\{f_0\circ g_0\st f_0\weireducible f \text{ and } g_0\weireducible g \}. \]
	Intuitively the compositional product captures the idea of applying first $g$, then using some computable operation on the output of $g$ and then feeding the result as an input for $f$. The fact that the compositional product is well-defined has been shown in \cite[Cor.\ 3.7]{BP16}. Notice that, by definition, $f*g$ is a Weihrauch degree and not a specific multivalued function. Nonetheless, with a small abuse of notation, we will write formulas such as $h\weireducible f\compproduct g$ with the obvious meaning. We also write $f^{[n]}$ to denote the $n$-fold compositional product of $f$ with itself, where $f^{[0]}:=\id{}$ and $f^{[1]}:=f$.

	We say that $f$ is a \textdef{cylinder} if $f\strongweiequiv \id{}\times f$, where $\id{}$ denotes the identity on the Baire space. The notion of cylinder is very useful because of the following
	\begin{unprovedproposition}[{\cite[Cor.\ 3.6]{BG09}}]
		For every multivalued functions $f$ and $g$, if $g$ is a cylinder then $f\weireducible g \iff f \strongweireducible g$.
	\end{unprovedproposition}
	Cylinders are also useful to deal with the compositional product, thanks to the so-called \textdef{cylindrical decomposition}:
	\begin{unprovedproposition}[{\cite[Lem.\ 3.10]{BP16}}]
		For all $f,g$ and all cylinders $F, G$ with $F \weiequiv f$ and $G \weiequiv g$ there exists a computable $K$ such that $f*g\weiequiv F\circ K \circ G$.  
	\end{unprovedproposition}
	In particular, knowing that for every function $f$ we have that $f\weiequiv \id{}\times f$ and that the latter is a cylinder, we can always take a representative of $f*g$ of the form $(\id{}\times f)\circ \Phi_e \circ (\id{}\times g)$ for some computable function $\Phi_e$.

	Finally we define the \textdef{total continuation} or \textdef{totalization} of $f$, written $\mathsf{T}f$, as the total multivalued function $\mathsf{T}f(x):=f(x)$ if $x\in\dom(f)$ and $\mathsf{T}f(x)=Y$ otherwise. Clearly $\mathsf{T}f = f$ iff $f$ is total. Notice that the definition of $\mathsf{T}f$ is sensitive to the particular definition of $f$ as a multivalued function between represented spaces. For a more detailed exposition we refer to \cite{BGCompOfChoice19}.

	Let us now define some well-known multivalued functions that will be useful in the development of the work: 
	\begin{itemize}
		\item $\LPO\colon\Baire \to \{0,1\}$ is defined as $\LPO(p):=1$ iff $(\exists n)(p(n)=0)$. It is often convenient to think of $\LPO$ as the problem of finding a yes/no answer to a $\Sigma^{0,p}_1$ or $\Pi^{0,p}_1$ question. 
		\item $\mflim\pfunction \infStrings{(\Baire)}\to\Baire:=\sequence{p_n}{n\in\mathbb{N}}\mapsto \lim_{n\to\infty} p_n$, where the domain of $\mflim$ is the set of Cauchy sequences.
		\item $\NHA\mfunction{\Baire}{\Baire}$ is defined as $\NHA(p)=\{q\st q\text{ is not hyperarithmetic in }p\}$.
		\item $\chiPi\colon\Baire\to\{ 0,1\}$ is the characteristic function of a $\Pi^1_1$-complete set. It is often convenient to think of $\chiPi$ as the function that takes in input a tree on $\mathbb{N}$ and checks whether it is well-founded.
	\end{itemize} 

	A central role is played by the choice problems: given a represented space $X$ we define $\codedChoice{\boldfaceGamma}{}{X}\pmfunction{\boldfaceGamma(X)}{X}$ as the multivalued function that chooses an element from a non-empty set $A\in \boldfaceGamma(X)$. If $\boldfaceGamma=\boldfacePi^0_1$ we simply write $\codedChoice{}{}{X}$. We also write $\codedUChoice{\boldfaceGamma}{}{X}$ if the choice is restricted to singletons. 

	Three choice problems that turned out to be very relevant to calibrate the strength of multivalued functions that corresponds to theorems around $\mathrm{ATR}_0$ (from the point of view of reverse mathematics) are $\UCBaire$, $\CBaire$ and $\TCBaire$. They have been explored in great detail in \cite{KMP20}. It is known that $\mflim^{(n)}\strictlyweireducible\UCBaire$ for every $n$ (see \cite[Sec.\ 6]{BdBPLow12}) and indeed $\mflim^\dagger \weiequiv \UCBaire$, where $\mflim^\dagger$ corresponds to the iteration of $\mflim$ over a countable ordinal (\cite{Pauly15ordinals}). Moreover $\UCBaire$ and $\CBaire$ are closed under compositional product (\cite[Thm.\ 7.3]{BdBPLow12}). In \cite{KMP20} it is proved that $\codedUChoice{\boldfaceSigma^1_1}{}{\Baire}\weiequiv\UCBaire$ and $\codedChoice{\boldfaceSigma^1_1}{}{\Baire}\weiequiv \CBaire$. The fact that $\UCBaire\strictlyweireducible\CBaire$ follows from the fact that the element of a $\Sigma^1_1$ singleton is hyperarithmetic, but the hyperarithmetic functions are not a basis for the $\Pi^0_1$ predicates (see \cite[Thm.\ I.1.6 and thm.\ III.1.1]{SacksHRT}). In particular we have 
	\begin{unprovedtheorem}[{\cite[Cor.\ 3.4]{KMP20}}]
		\thlabel{thm:ucbaire_hyperarithmetic_solution}
		Let $f\pmfunction{\Baire}{X}$ be a (partial) multivalued function, for some represented space $X$. If $f\weireducible \UCBaire$ then, for every $x\in\dom(f)$, $f(x)$ contains some $y$ hyperarithmetic relative to $x$.
	\end{unprovedtheorem}

	The Weihrauch degree of $\TCBaire$ has been explored in \cite[Sec.\ 8]{KMP20}. It is known that $\CBaire \strictlyweireducible \TCBaire$ (\cite[Prop.\ 8.2(1)]{KMP20}) and $\TCBaire^*$ is one of the strongest problem studied so far that is still considered among the ``$\mathrm{ATR}_0$ analogues''.

	We conclude this section with the following proposition.
	\begin{proposition}
		\thlabel{prop:product_with_NHA_general}
		Let $f\pmfunction{X}{Y}$ and $g\pmfunction{Z}{W}$ be multivalued functions between represented spaces and let $A\subset \dom(g)$ be s.t.\ 		
		\[ \{ z\in\dom(g)\st (\forall w\in g(z))( w\text{ is not hyperarithmetic in } z ) \} \subset A. \]
		If $f\times \NHA\weireducible g$ then $f\weireducible g\restrict{A}$. 
	\end{proposition}
	\begin{proof}
		Assume $f\times\NHA\weireducible g$ and let the reduction be witnessed by the computable functions $\Phi,\Psi$. For every $p_x$ which is the name of some $x\in\dom(f)$, the pair $(p_x,p_x)$ is mapped via $\Phi$ to a name $p_z$ for some element $z\in\dom(g)$.
		
		It suffices to show that for every name $p_x$ of an element in the domain of $f$, the pair $(p_x,p_x)$ is mapped via $\Phi$ to a name $p_z$ of an element $z\in A$.

		If this were not the case then, for some $x\in\dom(f)$, $p_z$ is the name of some $z\notin A$. By hypothesis, there is a $w\in g(z)$ s.t.\ $w$ has a name $p_w$ which is hyperarithmetic in $p_z$. Let $G$ be a realizer of $g$ s.t.\ $p_w=G\Phi(p_x,p_x)$. Since $p_w$ is hyperarithmetic in $p_z$, and hence in $p_x$, we have reached a contradiction with the fact that $\Psi(p_w,p_x,p_x)$ computes a solution for $\NHA(p_x)$.
	\end{proof}
	
	This result will often be used in combination with \thref{thm:ucbaire_hyperarithmetic_solution}. In fact if there is a computable $x\in \dom(f)$ s.t.\ $f(x)$ does not contain any hyperarithmetic element, then $f\not\weiequiv\UCBaire$.
	
\section{Ramsey theorems}
\label{sec:ramsey}

The space $\Ramsey$, endowed with the induced topology from the Baire space $\Baire$ is computably isometric to $\Baire$. There is actually a canonical choice for a computable bijection $\Baire\to\Ramsey$.

Let $P\subset\Ramsey$. We say that $f$ is a \textdef{homogeneous solution for $P$} iff
\[ (\forall g\in\Ramsey)(fg\in P) \lor (\forall g\in\Ramsey)(fg\notin P). \]
If $f$ is homogeneous for $P$ we say that $f$ \textdef{lands in $P$} if the first disjunct holds, i.e.\ if $(\forall g\in\Ramsey)(fg\in P)$. Vice versa, if $(\forall g\in\Ramsey)(fg\notin P)$ then we say that $f$ \textdef{avoids $P$}. A set $P\subset\Ramsey$ is called \textdef{Ramsey} (or we say that it has the \textdef{Ramsey property}) iff it has a homogeneous solution. We will denote the set of homogeneous solutions for $P$ (which may either land in it or avoid it) with $\HomSol(P)$. Notice that, in general, a set can have both solutions that land in the set and solutions that avoid the set.

In the literature, the symbol $\Ramsey$ is sometimes used to denote the family of all infinite subsets of $\mathbb{N}$. Also, if $X$ is an infinite subset of $\mathbb{N}$, $[X]^\mathbb{N}$ denotes the family of all infinite subsets of $X$. It is easy to identify the Ramsey space $\Ramsey$ with the space of infinite subsets of $\mathbb{N}$ (by identifying a function $f$ with its range). With this in mind, we may write $[f]^\mathbb{N}:=[\ran(f)]^\mathbb{N}$ to denote the set of all infinite subsequences of $f$. The definition of homogeneous solution can now be written as
\[ [f]^\mathbb{N}\subset P \lor [f]^\mathbb{N}\cap P = \emptyset. \]

It is natural to ask which classes of subsets of $\Ramsey$ have the Ramsey property. The problem is well studied and has an extensive literature. The Galvin-Prikry theorem (\cite{GalvinPrikry73}) states that all Borel subsets of $\Ramsey$ have the Ramsey property. This result can actually be extended to analytic sets (\cite{Silver70}). To go beyond the analytic sets we need axioms above ZFC (see e.g.\ \cite[Rem.\ VI.7.6, p.\ 240]{Simpson09}, \cite[pp.\ 1036--1037]{Kastanas83}). We will focus on Nash-Williams' theorem (\cite{NashWilliams65}), which states that open sets have the Ramsey property. This is also known as the open Ramsey theorem. It, in turn, implies the clopen Ramsey theorem (which is the restriction of Nash-Williams' theorem to clopen sets). As already mentioned, the open and clopen Ramsey theorems are known to be equivalent to $\mathrm{ATR}_0$ over $\mathrm{RCA}_0$ (see \cite[Thm.\ V.9.7]{Simpson09}).

\subsection{Some useful tools}
\label{subsec:tools}
Before formalizing the open and clopen Ramsey theorems in the context of
Weihrauch reducibility as multivalued functions, let us explicitly state
some properties of the set of homogeneous solutions that will turn out to be
useful in the rest of the paper. As a notational convenience we will use the
letters $P,Q,\hdots$ to denote open sets and $D,E,\hdots$ to denote clopen
sets.

\begin{proposition}
	\thlabel{prop:galvin_prikry_subset}
	Let $\boldfaceGamma$ be a definable (boldface) pointclass that is downward closed with respect to Wadge reducibility (i.e.\ it is a downward closed family of Wadge degrees), such as the families of open and of clopen sets. Assume that every $P\in\boldfaceGamma(\Ramsey)$ is Ramsey and that for every $f\in\Ramsey$,
	\[ \boldfaceGamma([f]^\mathbb{N}) = \{ P\cap [f]^\mathbb{N} \st P \in \boldfaceGamma(\Ramsey) \}. \]
	Then for every $f\in\Ramsey$, every $Q\in\boldfaceGamma([f]^\mathbb{N})$ is Ramsey. Moreover if $P\in\boldfaceGamma(\Ramsey)$ and $f\in\Ramsey$ there exists $h\in\HomSol(P)$ s.t.\ $h\substring f$.
\end{proposition}
\begin{proof}
	It is easy to see that every $f\in\Ramsey$ induces a $f$-computable homeomorphism $\varphi_f\colon\Ramsey\to[f]^\mathbb{N}$ defined as
	\[ \varphi_f(p):= n \mapsto f(p(n)). \]
	Notice also that 
	\begin{equation*}\tag{$\star$}
		\varphi_f(h)g = fhg = \varphi_f(hg).
	\end{equation*}
	In particular this homeomorphism preserves subsequences, i.e.\ for every $q\substring p$ we have $\varphi_f(q)\substring \varphi_f(p)$.
	Fix $f\in\Ramsey$ and let $Q\in\boldfaceGamma([f]^\mathbb{N})$. Since $\boldfaceGamma$ is closed under Wadge reducibility we have that
	\[ P:=\varphi_f^{-1}(Q)\in \boldfaceGamma(\Ramsey). \]
	Moreover, since every pointset in $\boldfaceGamma(\Ramsey)$ has the Ramsey property, there is $h\in \HomSol(P)$. Using $(\star)$, it is straightforward to conclude that $\varphi_f(h)\in \HomSol(Q)$.

	For the second part it suffices to apply the first part to $Q:=[f]^\mathbb{N}\cap P$, which is in $\boldfaceGamma([f]^\mathbb{N})$.
\end{proof}

\begin{proposition}
	\thlabel{prop:sets_with_disjoint_solutions}
	Let $A,B\subset \mathbb{N}$ be disjoint. Let $P,Q\in\boldfaceSigma^0_1(\Ramsey)$ be s.t.\
	\begin{enumerate}
		\item \label{itm:dsA} $(\forall f \in P)( f(0)\in A \text{ and } f(1) \in A);$
		\item \label{itm:dsB} $(\forall g \in Q)( g(0)\in B \text{ and } g(1) \in B).$
	\end{enumerate}
	If $R:=P \cup Q$ then
	\[ \HomSol(R)\cap R= (\HomSol(P)\cap P)\cup(\HomSol(Q)\cap Q). \]
\end{proposition}
\begin{proof}
	The inclusion $(\HomSol(P)\cap P)\cup(\HomSol(Q)\cap Q) \subset \HomSol(R)\cap R$ is trivial and always holds, so we only need to prove the converse direction. Let $h\in \HomSol(R)\cap R$ and assume that $h\in P$. By induction we can easily show that $\ran(h)\subset A$. Indeed, by point \ref{itm:dsA}, $h(0)\in A$ and $h(1)\in A$. Moreover, if $h(i)\in A$ then $h(i+1)\in A$ because $h$ lands in $R$: indeed if not then the substring $\coding{h(i), h(i+1), \hdots}$ of $h$ can neither be in $P$ nor in $Q$ (by the disjointness of $A$ and $B$), hence it cannot be in $R$, contradicting the fact that $h$ lands in $R$. This shows that $h\in P$ implies $h\in \HomSol(P)\cap P$. Notice also that, by the disjointness of $A$ and $B$, $\ran(h)\subset A$ implies $\ran(h)\cap B=\emptyset$ and therefore no subsequence of $h$ is in $Q$. Similarly we can show that $h\in \HomSol(R)\cap Q$ implies $h\in (\HomSol(Q)\cap Q)\setdifference (\HomSol(P)\cap P)$ and therefore the claim follows.
\end{proof}

The following construction was used by Avigad \cite{Avigad98} in his proof of the open Ramsey theorem in $\mathrm{ATR}_0$.
\begin{definition}
	Let $P\in\boldfaceSigma^0_1(\Ramsey)$ and let $\coding{P}$ be a name for $P$. We can define the tree
	\[ \closedsidesoltree{\coding{P}}:= \{ \sigma \in \incstring\st (\forall \tau\finiteSubstring \sigma)(\tau \not \in \coding{P} ) \}. \]		
\end{definition}
\begin{lemma}
\thlabel{lem:solution_closed_side_iff_path_through_tree}
Let $P\in\boldfaceSigma^0_1(\Ramsey)$. For every name $\coding{P}$ of $P$ and every $f\in\Ramsey$ we have
\[ f\in \HomSol(P)\setdifference P \iff f\in [\closedsidesoltree{\coding{P}} ]. \]
\end{lemma}
\begin{proof}
	If $f\notin [\closedsidesoltree{\coding{P}}]$ then $(\exists n )(f[n]\notin \closedsidesoltree{\coding{P}})$, i.e.\ $(\exists n)(\exists \tau\finiteSubstring f[n])(\tau \in \coding{P})$.
	This implies that there exists a $g\substring f$ s.t.\ $\tau \pprefix g$. This shows that $g\in P$ and hence $f\notin\HomSol(P)\setdifference{P}$.
	
	Let $f\in [\closedsidesoltree{\coding{P}}]$ and let $g\substring f$. If $g\in P$ then $(\exists n)(g[n]\in \coding{P})$, contradicting the fact that $f\in [\closedsidesoltree{\coding{P}}]$ (by definition of $\closedsidesoltree{\coding{P}}$). Therefore we have that $f\in \HomSol(P)\setdifference{P}$.
\end{proof}
Notice that the above lemma shows that $\HomSol(P)\setdifference P$ is closed whenever $P$ is open. In particular, if $D$ is clopen then $\HomSol(D)$ is closed: indeed, letting 
 $E:=\setcomplement[\Ramsey]{D}$, we have $\HomSol(D)=\HomSol(E)$ and 
\[\HomSol(D)=(\HomSol(D)\cap D)\cup (\HomSol(D)\setdifference D) = (\HomSol(E)\setdifference E) \cup (\HomSol(D)\setdifference D) \]
is the union of two closed sets.

On the other hand, the set of solutions for an open set $P$ that lands in $P$ can be $\Pi^1_1$-complete: let $\sequence{q_i}{i\in\mathbb{N}}$ be an enumeration of the rationals. We can define 
\[ T:=\{\sigma\in\incstring \st (\forall i<\abslength{\sigma}-1)(q_{\sigma(i+1)}<_\mathbb{Q} q_{\sigma(i)} ) \}. \]
A path through $T$ is an infinite descending sequence in $\mathbb{Q}$. If we define $P:=\setcomplement[\Ramsey]{[T]}$ we have that $\HomSol(P)\cap P$ is the set of well-suborders of $\mathbb{Q}$ (every suborder of $\mathbb{Q}$ that is not a well-order contains an infinite descending sequence with increasing indexes, and therefore, a subsequence that lands in $[T]$) and hence is $\Pi^1_1$-complete.

This underlines a critical difference between the problem of finding a homogeneous solution that lands in $P$ and finding one that avoids $P$. 

The following construction will be used in the following to move open sets around while ``preserving'' homogeneous solutions.
\begin{definition}
For every $n>1$ let $\operatorname{pow}_n:=i\mapsto n^{i+1}$. We
can define the map $\elementpower{n}\colon (\incstring \cup \Ramsey ) \to
(\incstring \cup \Ramsey )$ as
\[ \elementpower{n}(f):= \operatorname{pow}_n \circ f = i \mapsto n^{f(i)+1}. \]
It is clear that $\elementpower{n}$ is a computable injection with computable inverse.

Let $P\in\boldfaceSigma^0_1(\Ramsey)$ and let $\coding{P}$ be a name for $P$. We can define (with a small abuse of notation) 
	\[ \elementpower{n}{(\coding{P})}:=\bigcup_{\sigma\in\coding{P}} \{ f\in \Ramsey \st \elementpower{n}(\sigma)\pprefix f \}.  \]
We can naturally extend the definition to a multivalued map $\boldfaceSigma^0_1(\Ramsey)   \rightrightarrows \boldfaceSigma^0_1(\Ramsey)$ defining
	\[ \elementpower{n}(P):= \left\{ \elementpower{n}(\coding{P}) \st \coding{P} \text{ is a name of }P \right\}.  \]
\end{definition}

\begin{lemma}
	\thlabel{lem:elementpower_preserves_solutions}
	Let $P\in \boldfaceSigma^0_1(\Ramsey)$. Fix $n>1$ and let $Q:=\elementpower{n}(\coding{P})$ for some name $\coding{P}$ of $P$. Then 
	\[ f\in \HomSol(P)\cap P \Leftrightarrow \elementpower{n}(f) \in \HomSol(Q)\cap Q. \]
\end{lemma}
\begin{proof}
	It is straightforward to see that, for every $f\in \Ramsey$, $f\in P$ iff  $\elementpower{n}(f)\in Q$. Moreover, if $g\in\Ramsey$, then
		\[ \elementpower{n}(f)g = i \mapsto n^{fg(i)+1} = \elementpower{n}(fg). \]

	If $f$ is a homogeneous solution that lands in $P$ then $\elementpower{n}(f)g \in Q$ for every $g\in \Ramsey$, i.e.\ $\elementpower{n}(f)\in \HomSol(Q)\cap Q$. Vice versa, if $\elementpower{n}(f)\in \HomSol(Q)\cap Q$ then for every $g\in\Ramsey$ we have $\elementpower{n}(f)g=\elementpower{n}(fg)\in Q$, which implies $fg\in P$.		
\end{proof}

\begin{definition}
	Let $\sigma,\tau\in\incstring$. We define $\sigma\opencodepairing \tau$ to be the set of all strings of the form $\coding{(\rho(i),\theta(i))\st i<N }$ where $\rho,\theta\in\incstring$ and s.t.\ 
	\[ \sigma \prefix\rho \land \tau \prefix\theta \land N=\max\{\abslength{\sigma},\abslength{\tau}\}. \]
	Clearly the map $\opencodepairing$ can be extended to infinite strings by defining
	\[ f\opencodepairing g := \coding{ (f(0), g(0)), (f(1),g(1)), \hdots }. \]

	For the sake of readability, it is convenient to introduce the following notation: for $i=1,2$,  we define
	\[ \pi_i: (\mathbb{N}\times\mathbb{N})^{<\mathbb{N}}\cup (\mathbb{N}\times\mathbb{N})^{\mathbb{N}}\to \finBaire \cup \Baire \]
	as the map that, given in input a finite (resp.\ infinite) string of pairs, returns the finite (resp.\ infinite) string of the $i$-th elements of the pairs.

	Let $\coding{P},\coding{Q}$ be two names for two open subsets of $\Ramsey$. We can define
	\[ \coding{P}\opencodepairing \coding{Q} := \bigcup_{\sigma\in\coding{P}}\bigcup_{\tau\in\coding{Q}} \sigma \opencodepairing \tau, \]
	which is a name for a new open set.

	This leads to a map $\boldfaceSigma^0_1(\Ramsey)\times\boldfaceSigma^0_1(\Ramsey) \rightrightarrows \boldfaceSigma^0_1(\Ramsey)$ defined by
	\[ P\opencodepairing Q := \{ R\in\boldfaceSigma^0_1(\Ramsey)\st \coding{P}\opencodepairing\coding{Q} \text{ is a name for } R \}.  \]
\end{definition}

Notice that, in general, it is not true that if $f$ lies in the open set with name $\coding{P}\opencodepairing\coding{Q}$ then $\pi_i f\in \Ramsey$.

\begin{lemma}
	\thlabel{lem:opencodepairing_properties}
	Let $P_1,P_2\in\boldfaceSigma^0_1(\Ramsey)$ and let $\coding{P_1},\coding{P_2}$ be names for $P_1,P_2$ respectively s.t.\ every string in $\coding{P_1}$ has length at least $2$. Let $P\in\boldfaceSigma^0_1(\Ramsey)$ be the open set with name $\coding{P_1}\opencodepairing\coding{P_2}$. Then
	\[ \HomSol(P)\cap P = \{ f\opencodepairing g\st f\in \HomSol(P_1)\cap P_1 \text{ and } g\in \HomSol(P_2)\cap P_2\}. \]
\end{lemma}
\begin{proof}
	Notice first of all that $f\in \HomSol(P)\cap P$ implies that, for $i=1,2$,
	\[  \pi_i f \in \Ramsey.\]
	Indeed, fix $n\in\mathbb{N}$ and consider the substring $g:=\coding{f(n), f(n+1),\hdots}$ of $f$. Since $f$ is homogeneous we have $g\in P$. In particular, there is $\tau=\coding{(\tau_1(i),\tau_2(i))\st i<N } \in \coding{P_1}\opencodepairing\coding{P_2}$ s.t.\ $\tau\pprefix g$. Fix $\sigma_1\in \coding{P_1}$ and $\sigma_2\in\coding{P_2}$ s.t.\ $\tau \in \sigma_1\opencodepairing\sigma_2$. Since $\abslength{\sigma_1}\ge 2$ we have $\abslength{\tau}\ge 2$. Moreover
		\[ (\pi_i f)(n) = \tau_i(0) < \tau_i(1) =(\pi_i f)(n+1). \]
	Let now $f\in \HomSol(P)\cap P$. For every $g\substring f$ we have that $g\in \HomSol(P)\cap P$ and $\pi_1 g\in P_1$, $\pi_2 g \in P_2$ (indeed if $\pi_1 g \notin P_1$ or $\pi_2 g \notin P_2$ then $g\notin P$). Hence  $\pi_i f\in \HomSol(P_i)\cap P_i$ for $i=1,2$. The reverse inclusion is straightforward as if $f_i\in \HomSol(P_i)\cap P_i$ then $f:= f_1 \opencodepairing f_2$ is a homogeneous solution for $P$.
\end{proof}

The following generalizes the tree used by Solovay \cite{Solovay78}.

\begin{definition}
	\thlabel{def:solovay_open}
Let $\phi\colon\mathbb{N}\to\mathbb{N}$ be injective. For every tree $T\in \repincTree$ we define the \textdef{Solovay open set $\solovay{\phi}(T) $} as
\begin{align*}
	\solovay{\phi}(T) := \{ f\in \Ramsey \st {}&{} (\exists k)(\forall \tau \dominated f[k])(\tau \notin T ) \text{ and }\\
	& (\exists n,m\in\mathbb{N})(f(0)= \phi(n) \text{ and } f(1)=\phi(m))\}.
\end{align*}
If $\phi$ is the identity function we drop the subscript.
\end{definition}

It is easy to see that $\solovay{\phi}(T)$ is an open set.

\begin{lemma}
\thlabel{lem:solutions_solovay_open_set}
Let $T\in \repincTree$ and let $W:=\solovay{\phi}(T)$. Then 
\begin{enumerate}
	\item \label{itm:solovay_wf_tree} If $[T]=\emptyset$ then $\HomSol(W)\cap W\neq \emptyset$. Moreover, if $\phi$ is surjective then $W=\Ramsey$ and therefore $\HomSol(W) = \HomSol(W)\cap W = \Ramsey$.
	\item \label{itm:solovay_if_tree} If $[T]\neq\emptyset$ then $\HomSol(W) = \HomSol(W)\setdifference{W}$. Moreover every $f\in \HomSol(W)$ dominates a path through $T$.
\end{enumerate}
\end{lemma}
\begin{proof}
For each $f\in\Ramsey$ define the tree
\[ T_f := \{ \sigma \in T \st (\forall i<\abslength{\sigma})(\sigma(i)\le f(i)) \}.\]
\begin{description}
	\item[\ref{itm:solovay_wf_tree}] Notice that, for every $f\in\Ramsey$, the set $T_f$ is a finitely-branching well-founded subtree of $T$. By K\"onig's lemma, $T_f$ must be finite and therefore there is a $k$ s.t.\ every string in $T_f$ has length $<k$. This implies that every $\tau \dominated f[k]$ is not in $T$. If $\ran(f)\subset\ran(\phi)$ (as is always the case if $\phi$ is surjective) then $f\in\HomSol(W) \cap W$ .
	\item[\ref{itm:solovay_if_tree}] Notice that $\HomSol(W)\cap W=\emptyset$ because for every path $x\in [T]$ and every $f\in\Ramsey$, there exists $g\in\Ramsey$ that grows sufficiently quickly s.t.\ $x \dominated fg$ (as proved in \cite[p.\ 108]{Solovay78}). 
	
	Moreover, if $f\in \HomSol(W)\setdifference{W}$ then $(\forall k)(\exists \tau \dominated f[k])(\tau \in T)$. This implies that $T_f$ is infinite and therefore, by K\"onig's lemma, $[T_f]\neq\emptyset$. This concludes the proof as $[T_f]\subset [T]$.\qedhere
\end{description}
\end{proof}

\begin{definition}
Let $\phi\colon\mathbb{N}\to\mathbb{N}$. For every tree $T\in \repincTree$ we
define the clopen set $\describepath{\phi}(T)$ as
\begin{align*}
	\describepath{\phi}(T) := \{ f\in [\mathbb{N}]^{\mathbb{N}} \st (\exists \sigma_0,\sigma_1\in T) (\, & f(0)=\phi(\coding{\sigma_0}) \text{ and } \\
		& f(1)=\phi(\coding{\sigma_1}) \text{ and } \sigma_0\pprefix \sigma_1) \}. 
\end{align*}
If $\phi$ is the identity function we just drop the subscript.
\end{definition}

\begin{lemma}
	\thlabel{lem:clopen_describes_paths} Let $T\in\repincTree$ and let
$D:=\describepath{\phi}(T)$ for some computable strictly increasing function
$\phi\colon\mathbb{N}\to\mathbb{N}$. If $[T]\neq\emptyset$ then $\HomSol(D)\cap D\neq \emptyset$ and there is a uniform computable surjection $\HomSol(D)\cap D \to [T]$. On the other hand, if $[T]=\emptyset$
then $\HomSol(D)\cap D=\emptyset$. 
\end{lemma}
\begin{proof}
	Let $y\in [T]$ and define $h\in D$ s.t.\ $h(i)=\phi(\coding{y[i]})$. It is easy to see that $h$ is an homogeneous solution for $D$ landing in $D$, therefore $\HomSol(D)\cap D \neq \emptyset$. Moreover, given any $f$ that lands in $D$ we can compute a path $x$ through $T$ as
	\[ x := \bigcup_{k\in\mathbb{N}} \phi^{-1}f(k). \]
	Indeed, since $f$ lands in $D$ we have that $\phi^{-1}f(k)\in T$ for every $k\in\mathbb{N}$. Moreover, for every $k$, $\coding{f(k),f(k+1),\hdots}\in D$ so that $\phi^{-1}f(k+1)$ is a finite string that properly extends $\phi^{-1}f(k)$. Therefore $x$ is a well-defined function $\mathbb{N}\to\mathbb{N}$. Finally it is easy to see that $x\in [T]$: for every $i$, let $j>i$ s.t.\ there is a $k$ s.t.\ $x[j]=\phi^{-1}f(k)$. By definition of $D$ we have that $\phi^{-1}f(k)\in T$. Since $T$ is a tree, we can conclude that $x[i]\in T$. Notice that $\phi^{-1}$ is computable, therefore $x$ is computable from $f$. Notice also that, if $y$ and $h$ are as in the beginning of this proof, then $y = \bigcup_{k\in\mathbb{N}}\phi^{-1}h(k)$, which proves that the mapping is a surjection.

	On the other hand, if $[T]=\emptyset$ then, for every $f\in D$, there is an $i$ s.t.\ $\phi^{-1}(f(i))\notin T$ or $\phi^{-1}(f(i))\not\pprefix \phi^{-1}(f(i+1))$, otherwise we could compute a path through $T$. In any case if $f\in D$ then it is not a homogeneous solution for $D$.
\end{proof}

\begin{lemma}
\thlabel{lem:computable_operations}
The following maps are computable:
\begin{enumerate}
	\item \label{itm:comp_tree} $\boldfaceSigma^0_1(\Ramsey)\to\boldfacePi^0_1(\Ramsey):=P\mapsto\HomSol(P)\setdifference{P} $;
	\item \label{itm:elem_power} $ \elementpower{n}\mfunction{\boldfaceSigma^0_1(\Ramsey)}{\boldfaceSigma^0_1(\Ramsey)}:=P\mapsto\elementpower{n}(P)$, for every $n\in\mathbb{N}$;
	\item \label{itm:opencode_pairing} $\opencodepairing\mfunction{\boldfaceSigma^0_1(\Ramsey)\times\boldfaceSigma^0_1(\Ramsey)}{\boldfaceSigma^0_1(\Ramsey)}:=(P,Q)\mapsto P\opencodepairing Q$;
	\item \label{itm:solovay_closed} $\solovay{\phi}\colon\repincTree\to\boldfaceSigma^0_1(\Ramsey):=T\mapsto \solovay{\phi}(T)$, for every injective map $\phi\colon\mathbb{N}\to\mathbb{N}$ with computable range;
	\item \label{itm:path_description} $\describepath{\phi}\colon\repincTree\to\boldfaceDelta^0_1(\Ramsey):=T\mapsto \describepath{\phi}(T)$, for every invertible map $\phi\colon\mathbb{N}\to\mathbb{N}$ with computable inverse.
\end{enumerate}
\end{lemma}
\begin{proof}
${}$
\begin{description}
	\item[\ref{itm:comp_tree}] Let $P\subset \Ramsey$ be open and let $\coding{P}$ be a name for $P$. The definition of $\closedsidesoltree{\coding{P}}$ is computable in $\coding{P}$. Moreover, $x\in [\closedsidesoltree{\coding{P}}]$ iff $x\in \HomSol(P)\setdifference{P}$ (see \thref{lem:solution_closed_side_iff_path_through_tree}). Since a name for $A\in\boldfacePi^0_1(\Ramsey)$ can be a tree $T$ s.t.\ $A=[T]$, the claim follows.
	\item[\ref{itm:elem_power}, \ref{itm:opencode_pairing}] Straightforward from the definition.
	\item[\ref{itm:solovay_closed}] Let $T\in \repincTree$ be represented by its characteristic function $\charfun{T}$. We can define
	\begin{align*}
		\coding{\solovay{\phi}(T)} := \{ \sigma\in \incstring \st {}&{} (\forall \tau \dominated \sigma)(\tau \notin T) \text{ and }\\
		& (\exists n,m\in\mathbb{N})(\sigma(0)= \phi(n) \text{ and } \sigma(1)=\phi(m))\}.
	\end{align*}
	Notice that the universal quantifier is bounded, while the formula in the scope of the existential quantifier is equivalent to requiring that $\sigma(0)$ and $\sigma(1)$ are in $\ran(\phi)$, which is computable by hypothesis. Therefore $\coding{\solovay{\phi}(T)}$ is computable in $T$. 
	\item[\ref{itm:path_description}] Let $T\in \repincTree$. By definition of $\describepath{\phi}(T)$, the basic clopen cone $\{f\in\Ramsey\st \tau\pprefix f \}$ is a subset of $\describepath{\phi}(T)$ iff 
	\[ \phi^{-1}\tau(0)\in T \text{ and }\phi^{-1}\tau(1)\in T \text{ and }\phi^{-1}\tau(0)\pprefix \phi^{-1}\tau(1).\] 
	In particular, this shows that we can $T$-computably obtain open names for $\describepath{\phi}(T)$ and its complement. \qedhere
\end{description}
\end{proof}

\section{Ramsey theorems in the Weihrauch lattice}
\label{sec:rt_wei}
\subsection{Definitions}
There are several ways to formalize the open Ramsey theorem as a multivalued function.
\begin{definition}[Open Ramsey Theorem]
\label{def:openRT}
 We define the \textdef{full version} of the open Ramsey theorem as the (total) multivalued function
\[ \openRamsey\mfunction{\boldfaceSigma^0_1(\Ramsey)}{\Ramsey} := P \mapsto \HomSol(P).\]
We may modify the full version by adding the requirement on ``which side" we want the solution to be in. In this case, however, we need to restrict the domain to the family of open sets that admit a solution. We can define the \textdef{strong versions} of the open Ramsey theorem as the multivalued functions $\findHopen, \findHclosed \pmfunction{\boldfaceSigma^0_1(\Ramsey)}{\Ramsey}$ with domain respectively
\[ \dom(\findHopen):=\{P\in \boldfaceSigma^0_1(\Ramsey) \st \HomSol(P)\cap P \neq \emptyset\}, \]
\[ \dom(\findHclosed):=\{P\in \boldfaceSigma^0_1(\Ramsey) \st \HomSol(P)\setdifference P \neq \emptyset \} \]
and defined as $\findHopen(P):= \HomSol(P)\cap P$ and $\findHclosed(P):= \HomSol(P)\setdifference P$.
We may strengthen further the requirements, defining the \textdef{weak versions} of the open Ramsey theorem: namely we define $\wfindHopen$ as the restriction of $\findHopen$ to
\[ \dom(\wfindHopen):=\{ P \in \boldfaceSigma^0_1(\Ramsey) \st \HomSol(P)\subset P \}. \]
Similarly we can define the weak version of $\findHclosed$ as the multivalued function $\wfindHclosed$ obtained by restricting $\findHclosed$ to
\[ \dom(\wfindHclosed):=\{ P \in \boldfaceSigma^0_1(\Ramsey) \st \HomSol(P)\cap P =\emptyset \}. \]
\end{definition}

Recall that, in general, an open set can have both solutions that land in the set and solutions that avoid the set. The domain of $\wfindHopen$ (resp.\ $\wfindHclosed$) is therefore strictly smaller than the domain of $\findHopen$ (resp.\ $\findHclosed$). As we will see the two versions exhibit very different behaviors. Notice also that the weak versions are restrictions of $\openRamsey$, while the strong versions are not (the set of solutions can be strictly smaller).

As in the case of the open Ramsey theorem, we can consider different multivalued functions corresponding to the clopen Ramsey theorem. 
\begin{definition}[Clopen Ramsey Theorem]
We define the full version of the clopen Ramsey theorem as the multivalued function $\clopenRamsey\mfunction{\boldfaceDelta^0_1(\Ramsey)}{\Ramsey}:= D\mapsto \HomSol(D)$. 

The strong version of the clopen Ramsey theorem is the multivalued function
	\[ \findHclopen\pmfunction{\boldfaceDelta^0_1(\Ramsey)}{\Ramsey}:=D\mapsto \HomSol(D)\cap D\]
with domain
	\[ \dom(\findHclopen):=\{ D\in\boldfaceDelta^0_1(\Ramsey)  \st \HomSol(D)\cap D \neq \emptyset \}. \]

The weak version of the clopen Ramsey theorem is the multivalued function $\wfindHclopen\pmfunction{\boldfaceDelta^0_1(\Ramsey)}{\Ramsey}$ defined as the restriction of $\findHclopen$ to
	\[ \dom(\wfindHclopen):=\{ D \in \boldfaceDelta^0_1(\Ramsey)\st \HomSol(D) \subset D\}. \]
\end{definition}

Notice that we defined only one strong and one weak version of the clopen Ramsey theorem. This is because, using \thref{lem:computable_operations_open}.\ref{itm:compl_clopen_set}, it is straightforward to see that the other two are (strongly) Weihrauch equivalent to the ones we defined.

\subsection{Problems reducible to  \texorpdfstring{$\UCBaire$}{\mathsf{UC}}}
\label{subsec:ucbaire}

We show that $\wfindHopen$, $\wfindHclopen$ and $\clopenRamsey$ are all Weihrauch equivalent to $\UCBaire$. None of these principles are strongly Weihrauch equivalent to $\UCBaire$, as we will show in \thref{thm:strong_wei_characterization}.

\begin{lemma}
\thlabel{thm:open<ucbaire}
$ \wfindHopen \weireducible \UCBaire $.
\end{lemma}
\begin{proof}
	Since $\ATR \weiequiv \UCBaire$ \cite{KMP20} it suffices to prove that  $\wfindHopen\weireducible\ATR$. Here we use the formulation of $\ATR$ given in \cite{gohatr}, namely $\ATR$ is defined as the function $\ATR : \WO\times \Cantor \times \mathbb{N} \to \Cantor$ mapping a countable well-order $L$, a set $A\subset \mathbb{N}$ and (a code for) an arithmetic formula $\theta(n,W,V)$ to the unique set $Y\subset \mathbb{N}$ satisfying $\atrformula_\theta(L,Y,A)$, where $\atrformula_\theta(X,S,T)$ is the arithmetic formula defined in \cite[Def.\ V.2.2]{Simpson09}, asserting that $X$ is a linear order and $S$ is the result of iterating $\theta$ along $X$ with parameter $T$.
	The proof is the direct translation in the context of Weihrauch reducibility of the proof presented in \cite[Lem.\ V.9.4]{Simpson09}. More details on the construction can be found in the paper where the proof was first presented, i.e.\ \cite[Sec.\ 3]{Avigad98}. Let $P\in\dom(\wfindHopen)$. The proof consists of four steps:
	\begin{enumerate}
		\item \label{item:or1} build the tree $T:=\closedsidesoltree{\coding{P}}$ of homogeneous solutions that avoid $P$;
		\item \label{item:or2} build $\KB(T)$;
		\item \label{item:or3} via arithmetic transfinite recursion along $\KB(T)$, obtain a sequence of infinite sets $\sequence{U_\sigma}{\sigma\in T}$ s.t.\ $U_\sigma \setdifference U_\tau$ is finite whenever $\tau \le_{\KB(T)} \sigma$, and classify each $\sigma \in \incstring$ as ``good'' or ``bad'';  
		\item \label{item:or4} use this classification to build a solution $f$.
	\end{enumerate}
	Notice that steps \ref{item:or1} and \ref{item:or2} are computable (using \thref{lem:computable_operations}.\ref{itm:comp_tree}). For $\sigma\notin T$, we classify $\sigma$ as good if the shortest prefix of $\sigma$ which is not in $T$ belongs to $\coding{P}$, and bad otherwise. For $\sigma\in T$, to define $U_\sigma$ and classify $\sigma$ as good or bad, we first define an infinite set $V_\sigma$ as follows:
	\begin{itemize}
		\item if $\sigma$ is the minimum of $\KB(T)$ then $V_\sigma := \mathbb{N}$;
		\item if $\sigma$ is the successor of $\tau$ in $\KB(T)$ then $V_\sigma := U_\tau$;
		\item if $\sigma$ is a limit in $\KB(T)$ then we define $V_\sigma$ by diagonal intersection: we computably and uniformly find a sequence $\tau_j$ cofinal in $\sigma$. Define
			  \[ \begin{gathered} 	u_0 := \min U_{\tau_0}; \\
								  u_{i+1} :=\min \left\{ \bigcap_{j\le i} U_{\tau_j}\setdifference\{u_j\}\right\};\\
								  V_\sigma := \{ u_i\st i\in\mathbb{N}\}. \end{gathered} \]
	\end{itemize}
	It is easy to verify that $V_\sigma$ is defined by an arithmetic formula. Let
	\[V_\sigma^1:= \{ m\in V_\sigma \st \sigma\concat m \text{ is good}\},\]
	and similarly $V_\sigma^0:= \{ m\in V_\sigma \st \sigma\concat m \text{ is bad}\} = V_\sigma \setdifference V_\sigma^1$. Set
	\[ n \in U_\sigma :\Leftrightarrow ( |V_\sigma^1|=\infty \text{ and } n \in V_\sigma^1) \text{ or } (|V_\sigma^1|<\infty \text{ and } n \in V_\sigma^0 ). \]
	We now classify $\sigma$ as good if $V^1_\sigma$ is infinite, and bad otherwise.
	
	We can obtain the information about $\sequence{U_\sigma}{\sigma\in T}$ and the goodness (or badness) for each $\sigma\in T$ as a name for $Y\in \ATR(\KB(T),P,\theta)$, for an appropriate arithmetic formula $\theta$. 
	
	As in \cite{Simpson09,Avigad98}, one can show that $\coding{}$ is good and compute a solution $f\in \wfindHopen(P)$ from $Y$.
\end{proof}

\begin{lemma}
\thlabel{thm:ucbaire<clopen}
$ \UCBaire \weireducible \wfindHclopen$.
\end{lemma}
\begin{proof}
We follow the proof of the fact that the clopen Ramsey theorem implies $\text{ATR}_0$ over $\text{RCA}_0$ presented in \cite[Lem.\ V.9.6]{Simpson09}. We actually prove the reduction $\SigmaSep \weireducible \wfindHclopen$, as the equivalence $\SigmaSep\weiequiv \UCBaire$ has been proved in \cite{KMP20}.

Let $\sequence{(T^0_k,T^1_k)}{k\in\mathbb{N}}$ be a sequence of pairs of trees s.t.\ for all $k$ at most one of $T^0_k$ and $T^1_k$ has a path (i.e.\ the sequence is a valid input for $\SigmaSep$). Our goal is to find a set $Z$ s.t.\ if $T^0_k$ has a path then $k\in Z$ and if $T^1_k$ has a path then $k\notin Z$.

Following \cite{Simpson09}, we can uniformly compute from $\sequence{(T^0_k,T^1_k)}{k\in\mathbb{N}}$ a name for a clopen $D\in\dom(\wfindHclopen)$ s.t.\ for every $f\in \wfindHclopen(D)$, $f$ and $\sequence{(T^0_k,T^1_k)}{k\in\mathbb{N}}$ uniformly compute some $Z\in \SigmaSep(\sequence{(T^0_k,T^1_k)}{k\in\mathbb{N}})$.
\end{proof}

\begin{theorem}
\thlabel{cor:wfindclopen=ucbaire}
$\UCBaire \weiequiv \wfindHopen \weiequiv \wfindHclopen$.	
\end{theorem}
\begin{proof}
This follows from \thref{thm:open<ucbaire}, \thref{thm:ucbaire<clopen} and the fact that $\wfindHclopen$ is the restriction of $\wfindHopen$ to clopen sets.
\end{proof}

\begin{theorem}
\thlabel{thm:ucbaire=clopenramsey}
$\UCBaire \weiequiv \clopenRamsey$.
\end{theorem}
\begin{proof}
By \thref{cor:wfindclopen=ucbaire} it suffices to prove that $\wfindHclopen\weiequiv \clopenRamsey$. The
reduction $\wfindHclopen \weireducible \clopenRamsey$ is trivial (the former
is a restriction of the latter).

Let us prove the reverse reduction. By \thref{prop:galvin_prikry_subset}, for every open
$P\subset\Ramsey$, if $f\notin \HomSol(P)$ then there is a $g\substring f$
s.t.\ $g\in \HomSol(P)$. This implies that the set $\setcomplement[\Ramsey]{\HomSol(P)}$ has no homogeneous
solution that lies in itself.

Let $D\in \boldfaceDelta^0_1(\Ramsey)$ and fix a name $\coding{p,q}$ for $D$. Consider the set
\[ E:= \{ f\in\Ramsey \st (\exists \sigma,\tau\in p)(\sigma\concat\tau \pprefix f) \lor (\exists \sigma,\tau\in q)(\sigma\concat\tau \pprefix f)\}. \]
It is clear that $E$ is a clopen set and a name for $E$ is computable from $\coding{p,q}$.

Notice also that $\HomSol(D) \subset E$. Indeed, let $f$ be a homogeneous
solution for $D$ and assume first that $f\in D$. Since $D$ is open there
must be a $\sigma\in p$ s.t.\ $\sigma\pprefix f$. Moreover,
since $f$ is a homogeneous solution,
$g:=\coding{f(\abslength{\sigma}), f(\abslength{\sigma}+1),\hdots}$ must
again be in $D$, hence there must be a $\tau\in p$ s.t.\ $\tau \pprefix g$. This implies that $\sigma\concat\tau \pprefix f$, i.e.\
$f\in E$. The case $f\in \setcomplement[\Ramsey]{D}$ is analogous by
replacing $D$ with $\setcomplement[\Ramsey]{D}$.

Moreover we can notice that there are no homogeneous solutions that land in
$\setcomplement[\Ramsey]{E}$, i.e.\ $E\in\dom(\wfindHclopen)$. Indeed assume that $f$ avoids $E$ and let $\sigma,\tau\in\incstring$ be s.t.\ $\sigma\concat\tau \pprefix f$ and $\sigma\in p$ and $\tau\in q$ (the case in which $\tau\in p$ and $\sigma\in q$ is analogous). Let also $g:=\coding{f(\abslength{\sigma}), f(\abslength{\sigma}+1), \hdots}$. Since $g\substring f$ and $f$ is homogeneous there must be $\rho\in p$ s.t.\ $\tau\concat\rho \pprefix g$. We can now notice that the subsequence $h$ of $f$ defined as 
\[ h:= \coding{f(0),\hdots,f(\abslength{\sigma}-1), f(\abslength{\sigma}+\abslength{\tau}),f(\abslength{\sigma}+\abslength{\tau}+1),\hdots } \] 
is s.t.\ $\sigma\concat\rho\pprefix h$, and therefore $h\in E$, contradicting the fact that $f$ avoids $E$.

We now claim that $\HomSol(D) = \HomSol(E)$. Once the claim is proved we can
use $\wfindHclopen (E) = \clopenRamsey(D)$ to finish the reduction.

It is straightforward to notice that $\HomSol(D)\subset E$ implies $\HomSol(D)\subset\HomSol(E)$, hence we only need to prove the inclusion $\HomSol(E)\subset\HomSol(D)$. Let $f\in \HomSol(E)$. Since $E\in\dom(\wfindHclopen)$ we must have $f\in E$. Assume w.l.o.g.\ that $f\in D$ and let $\sigma\pprefix f$ be s.t.\ $\sigma\in p$. Fix $g\substring f$ and let $\rho\pprefix g$ be s.t.\ $\rho \in p \cup q$. Let $k\in \mathbb{N}$ be s.t.\ $\max \rho < f(k)$ and $\max \sigma < f(k)$, and consider $\tau \in p \cup q$ be s.t.\ $\tau \pprefix \coding{f(k), f(k+1), \hdots}$. By the homogeneity of $f$ we have that $\tau \in p$ (otherwise every substring of $f$ that begins with $\sigma\concat \tau$ would not be in $E$). Let $h\substring f$ be s.t.\ $\rho\concat \tau\pprefix h$. Again, by the homogeneity of $f$ we have that $\rho \in p$, hence $g\in D$. Since $g$ was  arbitrary, we have that $f\in \HomSol(D)$.
\end{proof}

\subsection{Problems reducible to \texorpdfstring{$\CBaire$}{\mathsf{C}}}
\label{subsec:cbaire}
Here we consider $\wfindHclosed$, $\findHclosed$ and $\findHclopen$.

\begin{theorem}
	\thlabel{thm:findclopen=findclosed=cbaire} 
	$\CBaire \strongweiequiv \findHclopen \strongweiequiv \findHclosed$.
\end{theorem}
\begin{proof}
	We first show that $\findHclosed\strongweireducible \CBaire$. Given a name $\coding{P}$ for some open set $P\in\dom(\findHclosed)$, by \thref{lem:computable_operations}.\ref{itm:comp_tree} we can compute a name for the closed set $\HomSol(P)\setdifference P$, which by hypothesis is nonempty. Therefore we can use $\CBaire$ to pick a solution.

	Since $\findHclopen\strongweireducible \findHclosed$ is trivial, it remains to show that $\CBaire \strongweireducible \findHclopen$.
	Let $T\subset\incstring$ be s.t.\ $[T]\neq \emptyset$ and let
	$D:=\describepath{}(T)$, i.e.
	\[ D = \{ f\in [\finBaire]^\mathbb{N}\st f(0)\in T \text{ and } f(1)\in T \text{ and } f(0)\prefix f(1) \}. \]
	Recall that $D$ is computable from $T$ (see	\thref{lem:computable_operations}.\ref{itm:path_description}). Moreover, by \thref{lem:clopen_describes_paths}, we have that $D\in\dom(\findHclopen)$ and that every $f\in\findHclopen(D)$ uniformly computes a path through $T$.
\end{proof}

\begin{proposition}
	\thlabel{prop:ucbaire<wfindHclosed}
	$\UCBaire \strictlyweireducible \wfindHclosed$.
\end{proposition}
\begin{proof}
	The reduction $\UCBaire \weireducible \wfindHclosed$ is straightforward
	knowing that $\UCBaire\weiequiv \wfindHclopen$
	(\thref{cor:wfindclopen=ucbaire}). The fact that the reduction is strict
	follows from \cite[Sec.\ 3]{Solovay78}. In particular, Solovay showed that
	there is an open set $W$ with computable code s.t.\ every
	homogeneous solution avoids $W$ (hence $W$ is a valid input for
	$\wfindHclosed$) and is neither $\Sigma^1_1$ nor $\Pi^1_1$ (in particular it
	is not hyperarithmetic), while every computable instance of $\UCBaire$ has an
	hyperarithmetic solution (\thref{thm:ucbaire_hyperarithmetic_solution}).
\end{proof}

\begin{proposition}
\thlabel{prop:wfindHclosed<=cbaire}
	$\wfindHclosed \weireducible \CBaire \weiequiv \CCantor \compproduct \wfindHclosed$.
\end{proposition}
\begin{proof}
The first reduction follows from \thref{thm:findclopen=findclosed=cbaire} as $\wfindHclosed$ is the restriction of $\findHclosed$ to a smaller domain.
The reduction $\CCantor \compproduct \wfindHclosed\weireducible\CBaire$ is straightforward from $\CCantor\weireducible\CBaire$, $ \wfindHclosed\weireducible \CBaire$ and $\CBaire$ is closed under compositional product. 

Finally the reduction $\CBaire \weireducible \CCantor \compproduct \wfindHclosed$ is suggested by the proof of the corollary in \cite[Sec.\ 3]{Solovay78}. In particular, given an ill-founded tree $T\subset\Ramsey$ we can computably define the open set $W:=\solovay{}(T)$ (\thref{def:solovay_open}). By \thref{lem:solutions_solovay_open_set}, $W\in\dom(\wfindHclosed)$ and every solution $f\in \wfindHclosed(W)$ dominates a path through $T$. Let $X$ be the subtree of $\finBaire$ of the strings that are dominated by $f$ and let $T_f:= T \cap X$. Since $\emptyset\neq [T_f]\subset [T]$, we can use $\codedChoice{}{}{[X]}([T_f])$ to compute a path through $T$. To conclude the proof it is enough to notice that $\CCantor \weiequiv \codedChoice{}{}{[X]}$ (\cite[Thm.\ 7.23]{BGP17}).
\end{proof}

Notice that, by the choice-elimination principle (\cite[Thm.\ 7.25]{BGP17}), if $Y$ is a computable metric space, $f\pfunction X \to Y$ is a single-valued function and $f\weireducible\CBaire$ then $f\weireducible \wfindHclosed$. 

Since we are not able to show the equivalence of $\wfindHclosed$ with any known principle, it is worth to study its properties.

\begin{proposition}
	\thlabel{prop:wfindclosed_closed_product}
	$\wfindHclosed \strongweiequiv  \wfindHclosed \times \wfindHclosed$.
\end{proposition}
\begin{proof}
	Notice that if $P,Q\in \dom(\wfindHclosed)$ then $P\cup Q \in \dom(\wfindHclosed)$. Indeed, for every $f\in \Ramsey$, by the open Ramsey theorem applied to $[f]^\mathbb{N}\cap P$ (see \thref{prop:galvin_prikry_subset}), there is a $g\substring f$ s.t.\ $[g]^\mathbb{N}\subset [f]^\mathbb{N}\cap P$ or $[g]^\mathbb{N}\cap [f]^\mathbb{N}\cap P = \emptyset$. In the first case we would have a contradiction as $[g]^\mathbb{N}\subset [f]^\mathbb{N}\cap P$ implies $[g]^\mathbb{N}\subset P$, i.e.\ $\HomSol(P)\cap P \neq \emptyset$, against $P\in\dom(\wfindHclosed)$. Therefore we have $[g]^\mathbb{N}\cap [f]^\mathbb{N}\cap P = [g]^\mathbb{N}\cap P = \emptyset$, i.e.\ $g\in\HomSol(P)\setdifference{P}$. With a similar argument, we can now apply the open Ramsey theorem to $[g]^\mathbb{N}\cap Q$ and conclude that there is a $h\substring g$ s.t.\ $h\in \HomSol(Q)\setdifference{Q}$. In particular $h\notin Q$ and $h\notin P$ (as $h\substring g$ and $g$ avoids $P$). Therefore $h$ is a subsequence of $f$ that is not in $P\cup Q$. Since $f$ was arbitrary, we have that $\HomSol(P\cup Q)\cap (P\cup Q)=\emptyset$. This shows that every homogeneous solution $f\in \HomSol(P\cup Q)$ avoids $P\cup Q$, and, in particular, avoids both $P$ and $Q$. Since the union is computable (see \thref{lem:computable_operations_open}.\ref{itm:open_union}) we can compute a solution for $(\wfindHclosed\times\wfindHclosed)(P,Q)$ by computing $f\in\wfindHclosed(P\cup Q)$ and returning two copies of $f$.
\end{proof}

Let $\sigmaCofChoice\pmfunction{\boldfaceSigma^1_1(\mathbb{N})}{\mathbb{N}}$ be the multivalued function that chooses an element from a non-empty $\boldfaceSigma^1_1$ cofinite subset of $\mathbb{N}$. It is equivalent to assume that the input is a $\boldfaceSigma^1_1$ final segment of $\mathbb{N}$. Indeed, given a cofinite $\boldfaceSigma^1_1$ set $A$ we can computably define the $\boldfaceSigma^1_1$ set 
\[ B:=\{ n\in\mathbb{N} \st (\forall m\ge n)(m\in A)\}, \]
which is a (non-empty) final segment of $\mathbb{N}$. Since $B\subset A$, choosing an element in $B$ yields a solution for $\sigmaCofChoice(A)$. With this in mind, we can assume that an input for $\sigmaCofChoice$ is a sequence $\sequence{T_m}{m\in\mathbb{N}}$ of trees s.t.\ there exists $k$ s.t.\ $[T_i]=\emptyset$ iff $i<k$. 

The problem $\pSigmaCofChoice$ has been studied in \cite{KiharaADauriacChoice} under the name $\boldfaceSigma^1_1\text{-}\mathsf{AC}^{\vphantom{g}\mathsf{cof}}_{\Baire}$. Moreover, \cite{KMP20} (implicitly) uses $\pSigmaCofChoice$ in the proof of Lemma 4.7 to separate $\boldfaceSigma^1_1\text{-}\mathsf{WKL}$ from $\parallelization{\codedChoice{\boldfaceSigma^1_1}{}{\mathbb{N}}}$. It is known that $\pSigmaCofChoice \strictlyweireducible \CBaire$ \cite[Thm.\ 3.34]{KiharaADauriacChoice}. Moreover, $\UCBaire \strictlyweireducible \pSigmaCofChoice$ (\cite{GPVDescSeq}).

\begin{theorem}
	$\pSigmaCofChoice\strongweireducible \wfindHclosed$.
\end{theorem}
\begin{proof}
	We will use strings $\sigma$ which are prefixes of an infinite string $f$ obtained by joining countably many strings $g_i$; we write $\sigma = \operatorname{dvt}(\tau_0,\hdots,\tau_n)$ if $\tau_i$ is the prefix of $g_i$ contained in $\sigma$. Formally $\sigma = \operatorname{dvt}(\tau_0,\hdots,\tau_n)$ iff 
	\begin{itemize}
		\item $n = \max\{i: \coding{i,0} < \abslength{\sigma}\}$,
		\item for each $i$, $\abslength{\tau_i} = \max\{j: \coding{i,j} < \abslength{\sigma}\}+1$,
		\item for each $\coding{i,j} < \abslength{\sigma}$, $\tau_i(j) = \sigma(\coding{i,j})$.
	\end{itemize}
	
	Let $\sequence{T_{n,m}}{n,m\in\mathbb{N}}$ be a double sequence of trees s.t.\ for every $n$ there is $k_n$ s.t.\ $[T_{n,m}]=\emptyset$ iff $m<k_n$. For every $n$ we can define 
	\[ T_n := \coding{}\cup \bigcup_{m\in\mathbb{N}} \coding{m}\concat T_{n,m}. \]	
	Notice that, by hypothesis, for every $n$ we have $[T_n]\neq\emptyset$. Moreover, if $f\in [T_n]$ then $f(0)\in \sigmaCofChoice( \sequence{T_{n,m}}{m\in\mathbb{N}} )$.
	Define also
	\[ T:= \{\sigma \in \finStrings{\mathbb{N}}\st \sigma = \operatorname{dvt}(\tau_0,\hdots,\tau_k) \land (\forall i\le k)(\tau_i\in T_i) \}. \]
	Notice that if $f_n\in [T_n]$ then $\coding{f_n}_{n\in\mathbb{N}}\in [T]$, hence $[T]\neq\emptyset$. Moreover, if $f\in [T]$ and $i\le j$ then, letting $f[i]=\operatorname{dvt}(\tau_0,\hdots,\tau_k)$ and $f[j]=\operatorname{dvt}(\rho_0,\hdots,\rho_h)$, for every $n\le k$ we have $\tau_n\prefix \rho_n$. Therefore
	\[ f\in [T] \iff f= \coding{f_n}_{n\in\mathbb{N}} \text{ and } (\forall n\in \mathbb{N})(f_n \in [T_n]). \]
	Let $W:=\solovay{}(T)$ be the Solovay open set for $T$. By \thref{lem:solutions_solovay_open_set}, $W\in \dom(\wfindHclosed)$ and every $h\in \wfindHclosed(W)$ dominates a path through $T$.
	
	To conclude the proof we notice that, if $f=\coding{f_n}_{n\in\mathbb{N}}\in [T]$ and  $h$ dominates $f$ then, for every $n$,
	\[ h(\coding{n,0})\ge f_n(0) \in \sigmaCofChoice( \sequence{T_{n,m}}{m\in\mathbb{N}} ).\] In particular $h(\coding{n,0})\in\sigmaCofChoice(\sequence{T_{n,m}}{m\in\mathbb{N}} )$.
\end{proof}

Let $\ATR_2\mfunction{\LO\times \Cantor\times \mathbb{N}}{\{0,1\}\times\Baire}$ be the two sided version of $\ATR$, defined in \cite[Def.\ 8.2 and prop.\ 8.9]{gohatr} as the multivalued function whose inputs are triples $(L,A,\theta)$ and the output is a pair $(i,Y)$ s.t.\ either $i=0$ and $Y$ is a $<_L$-infinite descending chain or $i=1$ and $Y$ is a (pseudo)hierarchy $\coding{Y_a}_{a\in L}$ s.t.\ for all $b\in L$, $Y_b=\{n\st \theta(n,\bigoplus_{a<_L b} Y_a, A )\}$. It is known that $\UCBaire\strictlyweireducible \ATR_2 \strictlyweireducible \CBaire$ (\cite[Cor.\ 8.5 and 8.7]{gohatr}). Jun Le Goh (personal communication) observed the following corollary:

\begin{corollary}
	\thlabel{thm:wfindclosed_not_reducible_atr2}
	$\wfindHclosed\not\weireducible \ATR_2$.
\end{corollary}
\begin{proof}
	The claim follows from the fact that $\CBaire\weireducible\CCantor\compproduct\wfindHclosed$ (\thref{prop:wfindHclosed<=cbaire}) while $\CBaire
	\not\weireducible\CCantor\compproduct\ATR_2$ (\cite[Cor.\ 8.5]{gohatr}).
\end{proof}

Let us denote with $\TwfindHopen$ the total continuation of $\wfindHopen$, i.e.\ the (total) multivalued function with domain $\boldfaceSigma^0_1(\Ramsey)$ defined as
\[ \TwfindHopen(P):= \begin{cases}
	\HomSol(P)	& \text{if } P\in \dom(\wfindHopen); \\
	\Ramsey		& \text{otherwise}.
\end{cases} \]
The following proposition underlines the gap between $\wfindHopen$ and $\wfindHclosed$ ($\wfindHopen\strictlyweireducible \wfindHclosed$ by \thref{thm:open<ucbaire} and \thref{prop:ucbaire<wfindHclosed}).
\begin{proposition}
	\thlabel{thm:twfindopen<=atr2}
	$\TwfindHopen \weireducible \ATR_2$, and hence $\wfindHclosed\not\weireducible\TwfindHopen$.
\end{proposition}
\begin{proof}
	Let $P\in\boldfaceSigma^0_1(\Ramsey)$ be an input for $\TwfindHopen$ and consider the tree $\closedsidesoltree{\coding{P}}$. We can computably build the linear order $\KB(\closedsidesoltree{\coding{P}})$. Let $\theta$ be the arithmetic formula defined in the proof of \thref{thm:open<ucbaire}. Notice that $\KB(\closedsidesoltree{\coding{P}})$ is not necessarily a well-order, as we are not assuming $P\in\dom(\wfindHopen)$ (i.e.\ there may be solutions that avoid $P$). We therefore need to modify $\theta$ by requiring that for every $\sigma \in \closedsidesoltree{\coding{P}}$, if either $V_\sigma$ is finite or $V_\sigma \setdifference U_\tau$ is infinite for some $\tau <_{\KB(\closedsidesoltree{\coding{P}})} \sigma$ (this cannot happen if $\KB(\closedsidesoltree{\coding{P}})$ is a well-order), then we set $U_\sigma = \mathbb{N}$ and we label $\sigma$ as bad. Let $\theta'$ be the modified formula. 

	Let $(i,Y)\in \ATR_2(\KB(\closedsidesoltree{\coding{P}}), P, \theta')$. If $i=0$ then $Y$ is a $<_{\KB(\closedsidesoltree{\coding{P}})}$-infinite descending sequence and $P\notin\dom(\wfindHopen)$; therefore any $f\in\Ramsey$ is a valid output for $\TwfindHopen(P)$. Suppose now that $i=1$, so that $Y$ is a (pseudo)hierarchy. By construction, $Y$ yields a labeling of each $\sigma\in\incstring$ as ``good'' or ``bad'', and, for each $\sigma \in \closedsidesoltree{\coding{P}}$, an infinite set $U_\sigma$ (see the proof of \cite[Lem.\ V.9.4]{Simpson09}). The classical proof shows that if $P\in\dom(\wfindHopen)$ then $\coding{}$ is good. In particular if $\coding{}$ is bad then we can immediately conclude that $P\notin\dom(\wfindHopen)$ (and, again, any $f\in\Ramsey$ is a valid solution for the original problem). On the other hand, 
	if $\coding{}$ is good then the modifications made to obtain $\theta'$ ensure that the intersection of finitely many $U_\sigma$ is infinite, and therefore we can compute $f\in \Ramsey$ following the construction of the classical proof. Notice that if $P\in\dom(\wfindHopen)$ then $f\in \HomSol(P)=\TwfindHopen(P)$. On the other hand, if $P\notin\dom(\wfindHopen)$ then $f\in\TwfindHopen(P)$ (trivially).

	The second part follows from \thref{thm:wfindclosed_not_reducible_atr2}.
\end{proof}

\subsection{Problems not reducible to \texorpdfstring{$\CBaire$}{\mathsf{C}}}
\label{subsec:not_cbaire}

Let us turn our attention to the last two remaining problems, namely $\openRamsey$ and $\findHopen$.

\begin{proposition}
	\thlabel{prop:tcbaire<=ccantor_comp_openRamsey}
	$\TCBaire\weireducible \CCantor\compproduct \openRamsey$ and $\chiPi\strictlyweireducible\LPO\compproduct \openRamsey$.
\end{proposition}
\begin{proof}
	We first show $\TCBaire\weireducible\CCantor\compproduct\openRamsey$. Let $T\subset\Ramsey$ be a tree and let $W:=\solovay{}(T)$ be the Solovay open set of $T$. Let also $f\in \openRamsey(W)$. By \thref{lem:solutions_solovay_open_set}, if $[T]\neq \emptyset$ then $\HomSol(W)\cap W=\emptyset$ and $f$ is a bound for some $x\in [T]$. On the other hand, if $[T]=\emptyset$ then $W=\Ramsey$ and $f$ is just an arbitrary infinite string.
	
	Let $X$ be the subtree of $\finBaire$ of the strings that are dominated by $f$. Notice that $\TChoice{[X]}\weiequiv\Choice{[X]}$. Indeed, to show that $\TChoice{[X]}\weireducible \Choice{[X]}$ we can notice that, given a tree $S\subset X$, we can computably define an ill-founded tree $R$ as follows: for each level $n$ we check whether $S$ has no nodes at level $n$. 
	If this happens for some $n$, we can (computably) extend $S$ to an ill-founded tree $R$. If this never happens then $R=S$. It is straightforward to see that $\Choice{[X]}([R])\subset \TChoice{[X]}([S])$. 

	Let $T_f:= T \cap X$. By \cite[Thm.\ 7.23]{BGP17}, $\CCantor\weiequiv\Choice{[X]}\weiequiv\TChoice{[X]}$, therefore we can use $\CCantor$ to compute a solution $h\in\TChoice{[X]}([T_f])$. Notice that $h\in \TCBaire([T])$. Indeed, if $[T]\neq \emptyset$ then $[T_f]\neq \emptyset$ and $h$ is a path through $T$.

	A simple modification of the above argument shows that $\chiPi\weireducible\LPO\compproduct\openRamsey$. In fact, we can see the tree $T$ as an input for $\chiPi$. If $f\in\openRamsey(\solovay{}(T))$ then $T_f$ is a finitely branching tree. Thus whether $T_f$ is finite is a $\Sigma^0_1$ question in $T_f$. We can therefore use $\LPO$ to check if $T_f$ is infinite and hence establish whether it is well-founded or not (by K\"onig's lemma, a finitely-branching tree is infinite iff it has a path). 
	
	The reduction is trivially strict as $\chiPi$ always has a computable output.
\end{proof}

It follows from \thref{thm:openRT_not_<=_sTCbaire_x_lim^n} that the reduction $\TCBaire\weireducible\CCantor\compproduct\openRamsey$ is actually strict.

\begin{corollary}
\thlabel{cor:openRamsey_not<_cbaire}
$\openRamsey \not\weireducible \CBaire$.
\end{corollary}
\begin{proof}
If $\openRamsey \weireducible \CBaire $ then
\[ \chiPi \weireducible \LPO\compproduct \openRamsey \weireducible \CBaire \compproduct \CBaire \weiequiv \CBaire, \]
contradicting the fact that $\chiPi \not\weireducible \CBaire$ (see \cite[Sec.\ 7]{KMP20}).
\end{proof}

\begin{corollary}
\thlabel{cor:wfindHclosed<openRT}
$\wfindHclosed\strictlyweireducible \openRamsey$.
\end{corollary}
\begin{proof}
The fact that $\wfindHclosed\weireducible\openRamsey$ is trivial since $\wfindHclosed$ is a restriction of $\openRamsey$ to a smaller
domain. The reduction is strict because
$\openRamsey\not\weireducible\CBaire$ (\thref{cor:openRamsey_not<_cbaire})
but $\wfindHclosed\weireducible\CBaire$ (\thref{prop:wfindHclosed<=cbaire}).
\end{proof}

\begin{definition}
	For every represented space $X$, we define the \textdef{strong total continuation of} $\mathsf{C}_X$ to be the multivalued function $\mathsf{sTC}_X\mfunction{\boldfacePi^0_1(X)}{2\times X}$ defined as
	\[ \mathsf{sTC}_X(A):=\{ (b,x) \in 2\times X \st (b=0 \rightarrow A=\emptyset) \land (b=1 \rightarrow x\in A) \}. \]
	In particular, for $X=\Baire$ (and analogously for $X=\Cantor$) we can think of $\sTCBaire$ as the total multivalued function that, given in input a tree, returns a string $\coding{b}\concat x$ s.t.\ $b$ codes whether the tree is well-founded or not and, if it is ill-founded, then $x$ is a path through $T$. 
\end{definition}

\label{thm:tcbaire<stcbaire} It is clear that $\TCBaire \strictlyweireducible \sTCBaire$ (the fact that the reduction is strict follows from $\chiPi\not\weireducible\TCBaire$ \cite[Cor.\ 8.6]{KMP20}, while obviously $\chiPi\weireducible\sTCBaire$). We can also notice the following:

\begin{corollary}
	\thlabel{thm:sTCBaire<=sTCCantor*openRT}
	$\sTCBaire \weireducible \sTCCantor \compproduct \openRamsey$.
\end{corollary}
\begin{proof}
	It suffices to repeat the proof of the first statement of \thref{prop:tcbaire<=ccantor_comp_openRamsey}, using $\sTCCantor$ in place of $\CCantor$.
\end{proof}
We will prove in \thref{thm:tcbaire_arit_<_openRT} that the above reduction is actually strict.

\begin{proposition}
	$\parallelization{ \TCBaire } \weiincomparable \sTCBaire$.
\end{proposition}
\begin{proof}
	The fact that $\sTCBaire\not\weireducible \parallelization{ \TCBaire }$ follows from the obvious observation that $\chiPi\weireducible\sTCBaire$, while $\chiPi\not\weireducible\parallelization{ \TCBaire }$ (see \cite[Cor.\ 8.6]{KMP20}).

	On the other hand, if $\parallelization{ \TCBaire } \weireducible \sTCBaire$ then, in particular, $\TCBaire\times \CBaire\weireducible \sTCBaire$. Since $\NHA\weireducible\CBaire$ (see e.g.\ \cite[Cor.\ 3.6]{KMP20}), by \thref{prop:product_with_NHA_general}, this implies that $\TCBaire\weireducible\sTCBaire\restrict{A}$, where $A$ is the set of non-empty closed sets of $\Baire$ with no hyperarithmetic member (notice that $\sTCBaire(\emptyset)$ has computable solutions). In particular, this implies that $\TCBaire\weireducible\CBaire$, contradicting \cite[Prop. 8.2.1]{KMP20}.
\end{proof}

We will now show that $\openRamsey\not\weireducible\TCBaire$. We actually prove a stronger result that will be useful in Section~\ref{sec:artihemtic_reducibility}.

\begin{theorem}
	\thlabel{thm:openRT_not_<=_sTCbaire_x_lim^n}
	For every $n\in\mathbb{N}$, $\openRamsey\not\weireducible \sTCBaire \times \mflim^{(n)}$.
\end{theorem}
\begin{proof}
	Let $(X,\repmap{X})$ be the represented space of computably open subsets of $\Ramsey$ with no arithmetic homogeneous solution, where $\repmap{X}$ is the restriction of $\repmap{\boldfaceSigma^0_1(\Ramsey)}$ to computable names. Let us define $\openRamsey_X\mfunction{X}{\Ramsey}$ as $\openRamsey_X(P):=\HomSol(P)$. 
	
	The reduction $\openRamsey_X\weireducible\openRamsey$ holds trivially, hence it is enough to prove that  
	\[ \openRamsey_X \not\weireducible \sTCBaire \times \mflim^{(n)}.\]
	Assume by contradiction that there is a reduction. Since $\mflim^{(n)}$ is a cylinder, we can assume that the reduction is a strong Weihrauch reduction. Let $\Phi_1, \Phi_2,\Psi$ be the maps witnessing the strong reduction, with $\Phi_1$ producing an input for $\sTCBaire$ and $\Phi_2$ producing an input for $\mflim^{(n)}$. Assume that there is an $P\in X$ s.t.\ $\Phi_1(\coding{P})$ is a name for the empty set, for some name $\coding{P}$ of $P$. By definition, $0^\omega$ is a valid output of $\sTCBaire(\emptyset)$. Let $q:= \mflim^{(n)}(\Phi_2(\coding{P}))$. Notice that $q$ is arithmetic, as $\coding{P}$ is computable by definition of $X$. We have now reached a contradiction as $\Psi(0^\omega, q)$ is arithmetic, against the fact that $P$ has no arithmetic solution.
	
	This implies that, for every $P\in X$ and every name $\coding{P}$ of $P$, $\Phi_1(\coding{P})$ is a name for a non-empty closed set, hence we have a reduction $\openRamsey_{X}\weireducible\CBaire \times \mflim^{(n)}\weiequiv \CBaire$.

	We now claim that $\openRamsey_X\not\weireducible\CBaire$, concluding the proof. We will in fact show that $\cchiPi\weireducible\LPO\compproduct\openRamsey_X$, where $\cchiPi$ is the restriction of $\chiPi$ to computable trees. The claim then follows from the fact that $\cchiPi\not\weireducible\CBaire$ (as $\cchiPi$ is not effectively Borel measurable, see \cite[Thm.\ 7.7]{BdBPLow12}) and the fact that $\CBaire$ is closed under compositional product.

	Let $\Phi_D$ be the forward functional witnessing $\codedUChoice{\boldfaceSigma^1_1}{}{\Baire}\weireducible\wfindHclopen$ (recall that $\codedUChoice{\boldfaceSigma^1_1}{}{\Baire}\weiequiv\UCBaire$ \cite[Thm.\ 3.11]{KMP20}, while $\UCBaire\weireducible\wfindHclopen$ has been proved in \thref{thm:ucbaire<clopen}). Let also $T_{NAR}$ be a computable input for $\codedUChoice{\boldfaceSigma^1_1}{}{\Baire}$ with a single non-arithmetic solution (recall that, by \cite[Thm.\ II.4.2]{SacksHRT} every $H$-set is a $\Pi^0_2$ singleton).
	
	Let $T$ be an input for $\cchiPi$. We can assume w.l.o.g.\ that $T$ has no hyperarithmetic path: indeed if $S$ is a computable ill-founded tree with no hyperarithmetic path then 
	\[ T\times S:=\{ \coding{ \coding{\sigma(0),\tau(0)},\hdots, \coding{\sigma(n-1),\tau(n-1)}} \st \sigma \in T \text{ and } \tau\in S \}\]
	is ill-founded iff $T$ is, and $T\times S$ has no hyperarithmetic path. 
	
	Let $W:=\solovay{}(T)$ be the Solovay open set for $T$ and let $Q$ be the clopen set with name $\Phi_D(T_{NAR})$. Notice that, since $Q\in\dom(\wfindHclopen)$, for every $f$ we can computably find a subsequence $g\substring f$ s.t.\ $g\in Q$.
	
	We can computably define $P:=W\cap Q$ (see \cite[Prop.\ 3.2.4]{Brattka05}). Since $W$ and $Q$ are computable then so is $P$. Let us show that $P$ does not have any arithmetic solution, which implies $P\in X$. We distinguish two cases:
	\begin{enumerate}
		\item $[T]=\emptyset$:  by \thref{lem:solutions_solovay_open_set} we have that $W=\Ramsey$, hence $P=Q$ and $\HomSol(P)=\HomSol(Q)$. Since every solution for $Q$ computes the non-arithmetic solution for $T_{NAR}$, $P$ does not have arithmetic solutions.
		\item $[T]\neq\emptyset$: notice first of all that $P\in\dom(\wfindHclosed)$ as $P\subset W$ and $W\in\dom(\wfindHclosed)$ (see \thref{lem:solutions_solovay_open_set}). 	
		
		Given $f\in\HomSol(P)$ then, by the above observation, we can computably find a subsolution $g\in\HomSol(P)$ s.t.\ $g\in Q$, thus $g\notin W$. By K\"onig's lemma such a $g$ is a bound for a path through $T$ (see the proof of \thref{lem:solutions_solovay_open_set}). This also implies that every $f\in\HomSol(P)$ is not (hyper)arithmetic (as, by hypothesis, $T$ does not have hyperarithmetic paths).
	\end{enumerate}
	Given $f\in \openRamsey_{X}(P)$ we can computably find $g\substring f$ s.t.\ $g\in Q$. Let $T_g$ be the subtree of $T$ bounded by $g$. Notice that $g$ is a bound for a path through $T$ iff $T_g$ is ill-founded iff $T$ is ill-founded (as shown in case $2$ above). Since $T_g$ is a finitely-branching tree, by K\"onig's lemma $T_g$ is ill-founded iff it is infinite. Moreover, the problem of checking whether $T_g$ is finite is a $\Sigma^{0,g}_1$ question, hence we can use $\LPO$ to solve the problem (as in the proof of \thref{prop:tcbaire<=ccantor_comp_openRamsey}).
\end{proof}

\begin{proposition}
	\thlabel{prop:product_with_NHA}
	For every (partial) multivalued function $f$, if $f\times\NHA \weireducible\openRamsey$ then $f\weireducible \CBaire$.
\end{proposition}
\begin{proof}
	Assume $f\times\NHA\weireducible \openRamsey$ and let
	$B:=\dom(\wfindHopen)$. Obviously $\openRamsey\restrict{B} = \wfindHopen$,
	hence, since $\wfindHopen\weiequiv\UCBaire$
	(\thref{cor:wfindclopen=ucbaire}), the restriction of $\openRamsey$ to $B$
	always has a solution that is hyperarithmetic relative to the input
	(\thref{thm:ucbaire_hyperarithmetic_solution}). Since $f\times
	\NHA\weireducible\openRamsey$, by \thref{prop:product_with_NHA_general}, we
	have that $f$ is reducible to the restriction of $\openRamsey$ to
	\[ A:= \setcomplement[\boldfaceSigma^0_1(\Ramsey)]{B} = \dom(\findHclosed).  \]
	This implies that $f\weireducible \findHclosed$, as for each $P\in A$ we have $\findHclosed(P)\subset \openRamsey(P)$ (and therefore every realizer for $\findHclosed$ is also a realizer for $\openRamsey\restrict{A}$). The claim follows from the fact that $\findHclosed\weiequiv\CBaire$ (\thref{thm:findclopen=findclosed=cbaire}).
\end{proof}

\begin{corollary}
	\thlabel{cor:openRamsey_<_cbaire_x_openRamsey}
	$\openRamsey\strictlyweireducible \NHA \times \openRamsey \weireducible \CBaire \times \openRamsey.$
\end{corollary}
\begin{proof}
	The first reduction is straightforward and the second one follows from the fact that $\NHA\weireducible\CBaire$ (see \cite[Cor.\ 3.6]{KMP20}). The fact that the first reduction is strict follows from \thref{prop:product_with_NHA} and the fact that $\openRamsey\not\weireducible\CBaire$ (\thref{cor:openRamsey_not<_cbaire}).
\end{proof}

To have a better understanding of the uniform strength of $\openRamsey$, we now show that, even with a parallel access to some hyperarithmetic computational power, $\openRamsey$ does not reach the level of $\TCBaire\times \CBaire$. Thus $\openRamsey$ is not at the level of $\TCBaire^*$, which is one of the strongest principles considered in \cite{KMP20} to be still at the level of $\mathrm{ATR}_0$.

\begin{proposition}
	\thlabel{thm:tcbaire_x_cbaire_not<_f_times_openRT}
	If $f\pmfunction{X}{Y}$ always has an hyperarithmetic solution relative to the input and $f\weireducible \CBaire$ then 
	$\TCBaire\times \CBaire \not \weireducible f \times \openRamsey$.
\end{proposition}
\begin{proof}
	Notice that, if we define $B:=\dom(\wfindHopen)$, then
	\[ (f \times\openRamsey)\restrict{X \times B} = f \times\wfindHopen. \]
	Since $\wfindHopen\weiequiv\UCBaire$ (\thref{cor:wfindclopen=ucbaire}) we have that $(f \times\openRamsey)\restrict{X \times B}$ always has a solution that is hyperarithmetic relative to the input (\thref{thm:ucbaire_hyperarithmetic_solution}). Assume by contradiction that the reduction $\TCBaire\times \CBaire \weireducible  f \times \openRamsey$ holds. By \thref{prop:product_with_NHA_general} we have that $\TCBaire$ is reducible to the restriction of $f \times \openRamsey$ to
	\[ A:= X \times (\setcomplement[\boldfaceSigma^0_1(\Ramsey)]{B}) = X \times\dom(\findHclosed). \]
	In particular this implies that $\TCBaire\weireducible f\times \findHclosed$ (see also the proof of \thref{prop:product_with_NHA}). We have therefore reached a contradiction as we would have
	\begin{equation*}
		\TCBaire\weireducible f \times\findHclosed \weireducible \CBaire\times \CBaire \weiequiv  \CBaire.	\qedhere
	\end{equation*}
\end{proof}
In particular, \thref{thm:tcbaire_x_cbaire_not<_f_times_openRT} implies $\TCBaire\times \CBaire \not \weireducible \UCBaire \times \openRamsey$. 
\medskip 

Let us now turn our attention to $\findHopen$. We first notice the following useful property:
\begin{proposition}
	\thlabel{prop:findHopen_closed_product}
	$\findHopen \times \findHopen \strongweireducible \findHopen$.
\end{proposition}
\begin{proof}
	Let $\coding{P_1},\coding{P_2}$ be names for two open sets $P_1, P_2\in \dom(\findHopen)$. Assume w.l.o.g.\ that every string $\sigma \in \coding{P_1}$ has length at least $2$ (there is no loss of generality as we can computably modify the code of $P_1$ by replacing a string with length $1$ with all its extensions of length $2$).
	
	Let $P$ be the open set with name $\coding{P_1}\opencodepairing \coding{P_2}$. Recall that $P$ is computable from $P_1$ and $P_2$ (see \thref{lem:computable_operations}). Moreover, by \thref{lem:opencodepairing_properties}
	\[ \HomSol(P)\cap P = \{ f\opencodepairing g\st f\in \HomSol(P_1)\cap P_1 \text{ and } g\in \HomSol(P_2)\cap P_2\}.  \]
	Since the projections $\pi_i$ are computable, it is clear that, from every solution of $\findHopen(P)$, we obtain two homogeneous solutions that land in $P_1$ and $P_2$ respectively.
\end{proof}

\begin{corollary}
	\thlabel{thm:findHopen_cylinder}
	$\findHopen$ is a cylinder.	
\end{corollary}
\begin{proof}
	This follows from \thref{prop:findHopen_closed_product} and the fact that $\idBaire \strongweireducible \findHopen$, as it then follows that
	\[ \idBaire \times \findHopen \strongweireducible \findHopen \times \findHopen \strongweiequiv \findHopen. \]
	To prove that $\idBaire\strongweireducible\findHopen$ we proceed as follows: let $p\in\Baire$ and assume w.l.o.g.\ that $p\in\Ramsey$. 
	Consider the tree $T:=\{ p[k]\st k\in\mathbb{N} \}$ of prefixes of $p$ and let $D:=\describepath{}(T)$. By \thref{lem:clopen_describes_paths} we have that $D\in\dom(\findHopen)$ and that every $f\in \HomSol(D)\cap D$ uniformly computes $p$.
\end{proof}

The problem $\findHopen$ is much stronger than all of the other Ramsey-related problems we introduced. We will in fact show that $\openRamsey\strictlyweireducible \findHopen$ (and this holds even if we consider arithmetic reductions, see \thref{thm:openRamsey_arith<_findHopen}). 

Although we will prove much stronger results it is worth it to sketch a short proof for the reduction $\openRamsey\weireducible\findHopen$. Given a name $\coding{P}$ for an open set $P$ build the open set 
\[ Q:=\elementpower{2}(\coding{P})\cup \describepath{\psi_3}(\closedsidesoltree{\coding{P}}),\]
where $\psi_3:=\sigma \mapsto 3^{\coding{\sigma}+1}$ and $\coding{\sigma}$ is the code of $\sigma$. Using \thref{lem:elementpower_preserves_solutions} and \thref{lem:clopen_describes_paths} one can prove that $Q\in\dom(\findHopen)$ and that every $f\in\HomSol(Q)\cap Q$ computes a solution for $P$. 

\begin{proposition}
\thlabel{prop:tcbaire<=findHopen}
$\sTCBaire \weireducible \findHopen$ and hence $\chiPi\strictlyweireducible\findHopen$.
\end{proposition}
\begin{proof}
Let $\stringcodepower{n}:\finBaire\to \mathbb{N}:=\sigma\mapsto n^{\coding{\sigma}+1}$, where $\coding{\sigma}$ is the code of $\sigma$.

Let $T\subset\incstring$ be a tree. We can define the open set $P\subset\Ramsey$ as $P:=P_1 \cup P_2$, where $P_1 := \describepath{\stringcodepower{2}}(T)$ and $P_2 := \solovay{\stringcodepower{3}}(T)$. 
Notice that, by \thref{lem:clopen_describes_paths} and \thref{lem:solutions_solovay_open_set} we have 
\[ [T]=\emptyset \iff \HomSol(P_1)\cap P_1= \emptyset \iff  \HomSol(P_2)\cap P_2\neq\emptyset. \]
Moreover, by \thref{prop:sets_with_disjoint_solutions},
\[ \HomSol(P)\cap P=(\HomSol(P_1)\cap P_1) \cup (\HomSol(P_2)\cap P_2).  \]
This implies that
\begin{gather*}
	[T]\neq\emptyset \Rightarrow \HomSol(P)\cap P = \HomSol(P_1)\cap P_1, \\
	[T]=\emptyset \Rightarrow \HomSol(P)\cap P = \HomSol(P_2)\cap P_2.
\end{gather*}
In particular, given a $f\in \HomSol(P)\cap P$ we can know whether $f\in \HomSol(P_1)\cap P_1$ or $f\in \HomSol(P_2)\cap P_2$ just by checking $f(0)$. If $f(0)$ is a power of $2$ then $[T]\neq\emptyset$ and we can compute a path through $T$ by considering the string $x\in \Ramsey$ s.t.\
\[ x = \bigcup_{i\in\mathbb{N}} \stringcodepower{2}^{-1}(f(i)).  \]
In the other case $[T]=\emptyset$ hence we can just return $\coding{0}\concat f$.

The result about $\chiPi$ follows from $\chiPi\weireducible\sTCBaire$ and $\sTCBaire\not\weireducible\chiPi$ as $\chiPi$ always has computable output.
\end{proof}

\begin{corollary}
	\thlabel{thm:tcbaire^*<findHopen}
	$\TCBaire^* \strictlyweireducible \findHopen$.
\end{corollary}
\begin{proof}
	The reduction follows from $\TCBaire\weireducible\findHopen$ (\thref{prop:tcbaire<=findHopen}) and the fact that $\findHopen$ is closed under product (\thref{prop:findHopen_closed_product}).
	The fact the reduction is strict follows from the fact that $\findHopen$ computes $\chiPi$ (\thref{prop:tcbaire<=findHopen}), while $\TCBaire^*$ does not (\cite[Cor.\ 8.6]{KMP20}).
\end{proof}

This shows that $\findHopen$ is properly stronger than any multivalued function arising from statements related to $\mathrm{ATR}_0$ studied so far.

\begin{theorem}
	\thlabel{thm:lim^n*openRT<=findHopen}
	$\CBaire\compproduct\openRamsey \weireducible \findHopen$.
\end{theorem}
\begin{proof}
	By the cylindrical decomposition we can write 
	\[ \CBaire\compproduct\openRamsey \weiequiv \CBaire\circ \Phi_e \circ (\id{}\times \openRamsey) \]
	for some computable function $\Phi_e$. It is enough to show that  
	\[ \CBaire\circ\Phi_e\circ (\id{}\times\openRamsey)\weireducible \findHopen\times \findHopen\times \chiPi \]
	and the claim will follow from $\chiPi\weireducible\findHopen$ (\thref{prop:tcbaire<=findHopen}) and the fact that $\findHopen$ is closed under product (\thref{prop:findHopen_closed_product}). 

	Let $\coding{p_1,p_2}$ be an input for $\CBaire\circ\Phi_e\circ (\id{}\times\openRamsey)$ and let $P$ be the open set with name $p_2$. We can consider the tree $\closedsidesoltree{p_2}$ of homogeneous solutions for $P$ that avoid $P$. We can now compute a tree $R$ s.t.\ for every $x,y\in\Baire$,
	\[x\in [\closedsidesoltree{p_2}] \text{ and } y\in [\Phi_e (p_1,x)] \iff \coding{x,y}\in [R].  \]
	Using the canonical computable bijection between $\Baire$ and $\Ramsey$ it is easy to transform $R$ into a tree $S\in \repincTree$ so that from any path through $S$ we can compute a path through $R$. 
	
	Recall that $\TwfindHopen\weireducible\CBaire$ (see \thref{thm:twfindopen<=atr2}). Since $\CBaire$ is closed under compositional product we have that $\CBaire\circ\Phi_e\circ (\id{}\times\TwfindHopen)\weireducible \CBaire$. Let $\Phi_A,\Psi_A$ be two computable maps witnessing the reduction. In particular, $\Phi_A(\coding{p_1,p_2})$ is an ill-founded subtree of $\finBaire$ and every path through $\Phi_A(\coding{p_1,p_2})$ computes a solution for $\CBaire\circ\Phi_e\circ (\id{}\times\TwfindHopen)$ via $\Psi_A$. Let also $\stringcodepower{n}$ be the function that maps $\sigma$ to $n^{\coding{\sigma}+1}$, where $\coding{\sigma}$ is the code of $\sigma$. 
	
	Let $D:= \describepath{\psi_2}(S)$ and define 
	\begin{gather*}
		U:= D \cup \describepath{\psi_3}(\Phi_A(\coding{p_1,p_2}));\\
		V:= D \cup \solovay{\psi_3}(S).
	\end{gather*}

	Let us first show that $U,V\in \dom(\findHopen)$. Notice that if $P\in\dom(\findHclosed)$ then $[\closedsidesoltree{p_2}]\neq \emptyset$ and $[S]\neq\emptyset$. By \thref{lem:clopen_describes_paths} we have that $\HomSol(D)\cap D\neq\emptyset$ and therefore $U,V\in\dom(\findHopen)$. On the other hand, assume $P\notin\dom(\findHclosed)$. Since $\TwfindHopen$ is total we have that $\Phi_A(\coding{p_1,p_2})$ is ill-founded. This implies that $\describepath{\psi_3}(\Phi_A(\coding{p_1,p_2}))$ has solutions that land in itself (again by \thref{lem:clopen_describes_paths}), and hence $U\in\dom(\findHopen)$. Moreover, since $P\notin\dom(\findHclosed)$ we have that $[S]=\emptyset$ and therefore $\HomSol(\solovay{\psi_3}(S))\cap \solovay{\psi_3}(S)\neq \emptyset$, which shows that $V\in\dom(\findHopen)$.

	Let $(f,g,b)\in (\findHopen\times \findHopen\times \chiPi)(U,V,S)$. We distinguish $2$ cases:
\begin{itemize}
	\item if $b=1$ then $[S]=\emptyset$, and hence $P\in \dom(\wfindHopen)$. By \thref{prop:sets_with_disjoint_solutions} $f$ lands in $\describepath{\psi_3}(\Phi_A(\coding{p_1,p_2}))$.
	
	In particular $f$ computes a path through $\Phi_A(\coding{p_1,p_2})$ (\thref{lem:clopen_describes_paths}). Moreover $\TwfindHopen(P)=\HomSol(P)=\openRamsey(P)$, so that $f$ computes also a solution for the compositional product (by applying $\Psi_A$ to the path).
	\item if $b=0$ then $[S]\neq\emptyset$ and hence $P\in\dom(\findHclosed)$. Moreover $\HomSol(V)\cap V = \HomSol(D)\cap D$. Indeed, by \thref{prop:sets_with_disjoint_solutions}, 
	\[ \HomSol(V)\cap V = (\HomSol(D)\cap D)\cup (\HomSol(\solovay{\psi_3}(S)) \cap \solovay{\psi_3}(S)) \]
	and $\HomSol(\solovay{\psi_3}(S)) \cap \solovay{\psi_3}(S)=\emptyset$ by \thref{lem:solutions_solovay_open_set}. 
	In this case $g$ computes a path through $S$, hence a path through $R$, and eventually a solution for the compositional product (by projecting the path through $R$). 
\end{itemize}
	The previous two points describe a way to compute a solution for the compositional product given a solution to $\findHopen\times \findHopen\times \chiPi$, and therefore conclude the proof.
\end{proof}

Notice that if $P\notin\dom(\wfindHopen)$ then we cannot (in general) use $U$ to compute a solution for the compositional product. Indeed, it may be that $\HomSol(P)\cap P\neq \emptyset$ and the solution obtained from $\findHopen(U)$ lands in $\describepath{\psi_3}(\Phi_A(\coding{p_1,p_2}))$. However, since every string is a valid solution for $\TwfindHopen(P)$, the solution we obtain is not guaranteed to have any connection with the original problem.

Notice moreover that $\openRamsey\strictlyweireducible\findHopen$ as the former is not closed under product with $\CBaire$ (\thref{cor:openRamsey_<_cbaire_x_openRamsey}) while the latter is closed under product (\thref{prop:findHopen_closed_product}) and computes $\CBaire$ (see \thref{prop:tcbaire<=findHopen}). We will prove a stronger result in \thref{thm:openRamsey_arith<_findHopen}.

\subsection{A \texorpdfstring{$0-1$}{0-1} law for strong Weihrauch reducibility}
\label{subsec:strong_wei}
We now characterize the strength of the Ramsey-related multivalued functions from the point of view of strong Weihrauch reducibility. 

\begin{proposition}
	\thlabel{thm:strong_wei_characterization}
	Let $\boldfaceGamma$ be a definable (boldface) pointclass that is downward closed with respect to Wadge reducibility. Assume also that every $P\in\boldfaceGamma(\Ramsey)$ is Ramsey and that, for every $h\in\Ramsey$,
	\[ \boldfaceGamma([h]^\mathbb{N}) = \{ P\cap [h]^\mathbb{N} \st P \in \boldfaceGamma(\Ramsey) \}. \]
	If $\mathsf{R}\pmfunction{\boldfaceGamma(\Ramsey)}{\Ramsey}$ is a multivalued function, s.t.\ for every $x\in\dom(\mathsf{R})$, $\mathsf{R}(x)=\HomSol(x)$, then
	\[ \id{2} \not\strongweireducible \mathsf{R}. \]
	In particular $\id{2}$ (and, a fortiori, $\UCBaire$) is not strongly Weihrauch reducible to $\wfindHopen$, $\wfindHclosed$, $\wfindHclopen$, $\openRamsey$, $\clopenRamsey$.
\end{proposition}
\begin{proof}
	Assume there is a strong Weihrauch reduction witnessed by the computable maps $\Phi,\Psi$. Let $p_i:=\Phi(i)$ (with a small abuse of notation we are identifying $i$ with its name) and let $P_i:=\repmap{\boldfaceGamma(\Ramsey)}(p_i)$. By definition of strong Weihrauch reducibility, for every $f\in\HomSol(P_i)$ we have $\Psi(f)=i$. Fix $f\in\HomSol(P_0)$ and consider the set $[f]^\mathbb{N}\cap P_1 \in \boldfaceGamma([f]^\mathbb{N})$. By \thref{prop:galvin_prikry_subset} we have that every pointset in $\boldfaceGamma([f]^\mathbb{N})$ has the Ramsey property, therefore there is a $g\substring f$ s.t.\ $[g]^\mathbb{N}\subset P_1 \cap [f]^\mathbb{N}$ or $[g]^\mathbb{N}\subset \setcomplement[{[f]^\mathbb{N}}]{P_1}$. In both cases $g\in\HomSol(P_1)$ and therefore $\Psi(g)=1$. However $g\substring f$, hence $g\in\HomSol(P_0)$ and so $\Psi(g)=0$, which is a contradiction.
\end{proof}

On the other hand, \thref{thm:findclopen=findclosed=cbaire} shows that $\CBaire\strongweiequiv \findHclosed \strongweiequiv \findHclopen$, which implies that $\findHclosed$ and $\findHclopen$ are cylinders. Since $\findHopen$ is also a cylinder (\thref{thm:findHopen_cylinder}) we have that, for every $g$ and every $f\in \{\findHopen,\findHclosed,\findHclopen \}$ 
\[ g\weireducible f \iff g \strongweireducible f. \]
This shows that, from the point of view of strong Weihrauch reducibility, the principles related to the open and clopen Ramsey theorems are either very weak (they do not strongly uniformly compute the identity on the $2$-element space) or they are as strong as possible (the notions of Weihrauch reducibility and strong Weihrauch reducibility coincide).

\section{Arithmetic Weihrauch reducibility}
\label{sec:artihemtic_reducibility}
Let us now define the notion of arithmetic Weihrauch reducibility, which is obtained by relaxing the computability requirements in the definition of Weihrauch reducibility. This was introduced in \cite[Def.\ 1.4]{GohThesis} (see also \cite[Def.\ 2.2]{GohWCWO} and \cite[Sec.\ 2.4]{KiharaADauriacChoice}).

\begin{definition}[Arithmetic Weihrauch reducibility]
Let $f\pmfunction{X}{Y}$, $g\pmfunction{Z}{W}$ be partial multivalued functions between represented spaces. We say that
$f$ is \textdef{arithmetically Weihrauch reducible} to $g$, and we write $f\weiarithreducible g$, if
\[ (\exists \text{ arithmetic } \Phi,\Psi\pfunction\Baire\to\Baire)(\forall G \realizer g )~ \Psi\coding{ (\idBaire, G\Phi) } \realizer f \]
where a function $F\pfunction\Baire\to\Baire$ is called \textdef{arithmetic} if there is $n\in\mathbb{N}$ s.t.\ $F\weireducible \mflim^{(n)}$.
\end{definition}
It is straightforward to see that $f\weireducible g \Rightarrow f \weiarithreducible g$. Notice moreover that if $f \weiarithreducible g$ then there exists $n$ s.t.\ $f\weireducible \mflim^{(n)}\compproduct g \compproduct \mflim^{(n)}$ (this follows directly from the definition of the compositional product).

\begin{proposition}
	\thlabel{thm:f_arith<_id}
	For every multivalued function $f$
	\[ (\exists n)(f\weireducible \mflim^{(n)}) \Rightarrow f\weiarithreducible \id{}. \]
\end{proposition}
\begin{proof}
	Assume there is a strong reduction $f\strongweireducible\mflim^{(n)}$ witnessed by the computable maps $\Phi_f,\Psi_f$. It is easy to see that the maps $\Phi:= \Psi_f\circ \mflim^{(n)} \circ \Phi_f$ and $\Psi:= \id{}$ witness the reduction $f\weiarithreducible \id{}$.
\end{proof}

\begin{corollary}
\thlabel{prop:ccantor_arith_equiv_id}
$\id{}\weiarithequiv \CCantor  \weiarithequiv \LPO \weiarithequiv \mflim^{(n)}$.
\end{corollary}
\begin{proof}
	Straightforward from \thref{thm:f_arith<_id} and the fact that $\id{}$ is Weihrauch reducible to $\CCantor$, $\LPO$ and $\mflim^{(n)}$.
\end{proof}

\begin{proposition}
\thlabel{prop:compprod_arith_reduction}
For every (partial) multivalued functions $f,g$, if $f\weiarithreducible \id{}$ then $f\compproduct g \weiarithequiv g\compproduct f \weiarithequiv g$.
\end{proposition}
\begin{proof}
Let us first prove $f\compproduct g \weiarithequiv g$, the other equivalence is analogous. We only need to prove that $f\compproduct g \weiarithreducible g$ as the converse reduction is trivial. We can assume w.l.o.g.\ that $f,g$ are (partial) multivalued functions $\pmfunction{\Baire}{\Baire}$ (see e.g.\ \cite[Lem.\ 3.8]{BGP17}).
By the cylindrical decomposition, we can write 
	\[ f\compproduct g \weiequiv (\id{} \times f) \circ \Phi_e \circ (\id{}\times g)  \]
for some computable $\Phi_e$. In particular
	\[ (\id{}\times f )\circ \Phi_e \circ (\id{}\times g)(\coding{p_1,p_2}) = \coding{\Phi_1(p_1, g(p_2)),f\circ \Phi_2(p_1, g(p_2)) }\]
where $\Phi_1,\Phi_2$ are the computable functions s.t.\ $\Phi_e(p)=\coding{\Phi_1(p),\Phi_2(p)}$.

Let $\Phi_f, \Psi_f$ be two arithmetic maps witnessing the reduction $f \weiarithreducible \id{}$. It is straightforward to see that the maps
\[ \Phi := \coding{p_1,p_2}\mapsto p_2, \]
\[ \Psi := (\coding{p_1,p_2}, q) \mapsto \coding{\Phi_1(p_1, q),\Psi_f(\Phi_2(p_1,q), \Phi_f \Phi_2(p_1, q)) }\]
witness the reduction $(\id{} \times f) \circ \Phi_e \circ (\id{}\times g) \weiarithreducible g$.
\end{proof}

\begin{corollary}
\thlabel{cor:wfindHclosed_arith_equiv_cbaire}
$\wfindHclosed \weiarithequiv \CBaire$.
\end{corollary}
\begin{proof}
This follows from $\CBaire \weiequiv \CCantor \compproduct \wfindHclosed$ (\thref{prop:wfindHclosed<=cbaire}), using \thref{prop:ccantor_arith_equiv_id} and \thref{prop:compprod_arith_reduction}.
\end{proof}

\begin{lemma}
	\thlabel{lem:arit_reduction_implies_normal_reduction}
	Let $g$ be a (partial multivalued) function that computes every arithmetic function and is closed under compositional product. For every (partial) multivalued function $f$ 
	\[ f\weiarithreducible g \Rightarrow f \weireducible g. \]
\end{lemma}
\begin{proof}
	It is enough to notice that $f\weiarithreducible g$ implies that there exists $n$ s.t. 
	\[ f\weireducible \mflim^{(n)}\compproduct g \compproduct \mflim^{(n)}.\]
	The hypotheses on $g$ immediately yield the claim.	
\end{proof}

\begin{corollary}
	\thlabel{thm:cbaire_arith<_tcbaire}
$\CBaire\strictlyweiarithreducible\TCBaire$.
\end{corollary}
\begin{proof}
The fact that $\CBaire\weiarithreducible \TCBaire$ is trivial as $\CBaire\weireducible\TCBaire$. The separation follows from \thref{lem:arit_reduction_implies_normal_reduction} (recall that, for every $n$, $\mflim^{(n)}\weireducible \UCBaire$, see \cite[Prop.\ 7.50]{BGP17}) and the fact that $\TCBaire\not\weireducible\CBaire$.
\end{proof}

\begin{theorem}
	$\TCBaire\weiarithequiv \sTCBaire$.
\end{theorem}
\begin{proof}
	This follows from \thref{thm:f_arith<_id}, \thref{prop:compprod_arith_reduction} and the fact that 
	\begin{equation*}
	\TCBaire \weireducible \sTCBaire \weireducible \LPO\compproduct\TCBaire.\qedhere
	\end{equation*}
\end{proof}

We will now prove the fact that $\TCBaire\strictlyweiarithreducible\openRamsey$. To do so we will first need some additional results about compositional products of iterations of $\mflim$ and $\TCBaire$.

\begin{lemma}
	\thlabel{thm:pi11_predicate}
	Let $D(X,Y, Z)$ be an arithmetic predicate with free variables among $X,Y,Z$ and let $\Phi\colon\Baire\times\Baire\to \repTree$ be computable. Define the $\Pi^1_1$ predicate $P(X, Y, Z)$ as
	\[ D(X,Y, Z) \land ([\Phi(X,Y)]\neq\emptyset \rightarrow Z\in[\Phi(X,Y)]). \]
	There exists a $\Pi^0_1$ predicate $S(X, Y, Z, W)$ s.t.\ an index for $S$ is computable from indices for $D$ and $\Phi$ s.t. 
	\begin{gather*}
		(\exists W)( S(X, Y,Z,W) )\Rightarrow D(X,Y,Z) , \\
		[\Phi(X,Y)]\neq \emptyset \Rightarrow (~ P(X,Y,Z) \iff (\exists W)(S(X, Y,Z,W))~).
	\end{gather*}
\end{lemma}
\begin{proof}
	By Kleene's normal form theorem (see e.g.\ \cite[Thm.\ 16.IV]{Rogers67}), there is a $\Pi^0_1$ predicate $T$ s.t.\
	\[ D(X,Y,Z) \iff (\exists W)(T(X,Y,Z,W)). \]
	Define the $\Pi^0_1$ predicate $S(X,Y,Z,W):= T(X,Y,Z,W) \land Z\in [\Phi(X,Y)]$. It follows from Kleene's normal form theorem that an index for $S$ is computable from indices for $D$ and $\Phi$. The first property of $S$ is immediate. For the second notice that, if $[\Phi(X,Y)]\neq \emptyset$ then
	\begin{align*}
		P(X,Y,Z)  & \iff D(X,Y,Z) \land Z\in [\Phi(X,Y)]\\ 
			&  \iff (\exists W)(T(X,Y,Z,W) \land Z\in [\Phi(X,Y)]) \\
			&  \iff (\exists W)(S(X,Y,Z,W)).\qedhere
	\end{align*}
\end{proof}

The previous lemma can be interpreted as follows: the predicate $P$ describes the compositional product (on both sides) of $\TCBaire$ with an arithmetic problem $f$, while $D$ says that $Y$ is a solution for $f(X,Z)$. Notice that, if we are considering the composition $\TCBaire\compproduct f$ then $f$ (and therefore $D$) will not depend on the output $Z$ of $\TCBaire$. On the other hand, if we consider $f\compproduct \TCBaire$ then we need to keep track of $Z$. The lemma proves that there is a uniform way to build a tree (whose body is the set of solutions to $S$) s.t., by projecting its paths, we can obtain the solutions to the original problem $P$. Notice however that the lemma does not guarantee such a tree to be ill-founded. In other words, we can recover (some) solutions to the original problem only if the tree is ill-founded.

Obviously, if $D$ depends only on $X,Y$ and not on $Z$, then a solution for $D$ can be (arithmetically) computed without first finding a path through the tree $\Phi(X,Y)$. 

\begin{lemma}
	\thlabel{thm:tcbaire*lim=tcbaire_x_lim}
	For every $n\in\mathbb{N}$, $\TCBaire\compproduct \mflim^{(n)} \weiequiv \TCBaire\times \mflim^{(n)}.$
\end{lemma}
\begin{proof}
	Fix $n\in\mathbb{N}$. The reduction $\TCBaire\times \mflim^{(n)}\weireducible \TCBaire\compproduct\mflim^{(n)}$ trivially follows from the algebraic rules of the operations (see \cite[Prop.\ 4.4]{BP16}). 
	
	To prove the converse reduction, by the cylindrical decomposition we can write 
	\[ \TCBaire\compproduct\mflim^{(n)}\weiequiv (\id{} \times \TCBaire)\circ \Phi_e \circ \mflim^{(n)}, \]
	for some computable function $\Phi_e$. In particular
	\[ (\id{} \times \TCBaire)\circ \Phi_e \circ \mflim^{(n)}(p) = \coding{\Phi_1(\mflim^{(n)}(p) ), \TCBaire \Phi_2(\mflim^{(n)}(p) )} \]
	where $\Phi_1,\Phi_2$ are the computable functions s.t.\ $\Phi_e(p)=\coding{\Phi_1(p),\Phi_2(p)}$. 
	
	Let $D(X,Y)$ be the predicate that says
	\[ Y = \coding{\Phi_1(\mflim^{(n)}(X) ),  \Phi_2(\mflim^{(n)}(X) )}. \]
	Notice that an index for $D$ can be (uniformly) computed from an index of $\Phi_e$. Define also the predicate $P(X,Y,Z)$ as
	\[ \quad D(X,Y) \land ([\pi_2 (Y)]\neq\emptyset\rightarrow Z\in [\pi_2 (Y)]), \]
	where $\pi_2:=\coding{Y_1,Y_2}\mapsto Y_2$. Since $D(X,Y)$ is arithmetic, we can use \thref{thm:pi11_predicate} to define a computable tree $S$ s.t.\ 
	\begin{gather*}
		(\exists W)( (X, Y,Z,W)\in[S] )\Rightarrow D(X,Y) , \\
		[\pi_2(Y)]\neq \emptyset \Rightarrow (~ P(X,Y,Z) \iff (\exists W)((X, Y,Z,W)\in [S] )~).
	\end{gather*}

	For every fixed $p\in \dom((\id{} \times \TCBaire)\circ \Phi_e \circ \mflim^{(n)})$ we define $S_p:=\{ \sigma \st \coding{p[\abslength{\sigma}],\sigma}\in S \}$. We now claim that, from an answer to $(\TCBaire\times\mflim^{(n)})([S_p],p)$ we can compute a solution to $(\id{} \times \TCBaire)\circ \Phi_e \circ \mflim^{(n)}(p)$. 
	
	Indeed $\Phi_1\circ \mflim^{(n)}(p)$ is trivially uniformly computed from $\mflim^{(n)}(p)$. On the other hand, there is a unique $q$ s.t.\ $D(p,q)$. Assume $[\pi_2(q)]\neq\emptyset$ and let $z_0 \in[\pi_2(q)]$. Since $P(p,q,z_0)$ holds, we have that $(\exists w)(\coding{q,z_0,w}\in [S_p])$, in particular $[S_p]\neq\emptyset$. Let $\coding{y,z,w}\in\TCBaire([S_p])$. Notice that, since $\mflim^{(n)}$ is single-valued, we have $y=q$. Hence we can conclude that $P(p,q,z)$ holds, and therefore, by projecting $\coding{y,z,w}$ on the second component we obtain a path through $[\pi_2(q)]$. If, on the other hand, $[\pi_2(q)]=\emptyset$, then any $z$ belongs to $\TCBaire([\pi_2(q)])$. In both cases, by projecting the output of $\TCBaire([S_p])$ we can compute a solution to $(\TCBaire\circ \Phi_2\circ \mflim^{(n)})(p)$ and this concludes the proof.
\end{proof}

\begin{lemma}
	\thlabel{thm:lim^n*tcbaire_<_lpo*tcbaire_x_lim^k}
	For every $n\in\mathbb{N}$, $\mflim^{(n)}\compproduct\TCBaire \weireducible \sTCBaire \times \mflim^{(3n+5)}.$
\end{lemma}
\begin{proof}
    Fix $n\in\mathbb{N}$. By the cylindrical decomposition we can write 
	\[ \mflim^{(n)}\compproduct\TCBaire\weiequiv \mflim^{(n)}\circ \Phi_e \circ (\id{} \times\TCBaire) \]
	for some computable function $\Phi_e$. 
	
	Let us define $F\mfunction{\Baire\times\Baire}{\Baire}$ as $F:= \mathsf{T}(\mflim^{(n)}\circ \Phi_e)$. Recalling that $\mflim^{(n)}=\mflim^{[n+1]}$, it is immediate that being in the domain of $\mflim^{(n)}$ is a $\Pi^0_{2n+3}$ property. On the other hand, whether $\Phi_e(p,q)$ is defined is a $\Pi^0_2$ property. This implies that $F\weireducible \mflim^{[2n+5+n+1]}=\mflim^{(3n+5)}$ and hence to prove the lemma it suffices to show that $ \mflim^{(n)}\circ \Phi_e \circ (\id{} \times\TCBaire)\weireducible \sTCBaire \times F$.

	Let $D(X,Y,Z)$ be the arithmetic predicate
	\[ Y = \mflim^{(n)}\circ \,\Phi_e(\pi_1(X), Z). \]
	Clearly an index for $D$ is computable from an index of $\Phi_e$. Let also $P(X,Y,Z)$ be the predicate
	\[ D(X,Y,Z) \land ([\pi_2(X)]\neq\emptyset \rightarrow Z\in [\pi_2(X)]). \]

	Since $D(X,Y,Z)$ is arithmetic, we can use \thref{thm:pi11_predicate} to define a computable tree $S$ s.t.\ 
	\begin{gather*}
		(\exists W)( (X, Y,Z,W)\in[S] )\Rightarrow D(X,Y,Z) , \\
		[\pi_2(X)]\neq \emptyset \Rightarrow (~ P(X,Y,Z) \iff (\exists W)((X, Y,Z,W)\in [S] )~).
	\end{gather*}

	For every fixed $p=\coding{p_1,p_2}\in \dom(\mflim^{(n)}\circ \Phi_e \circ (\id{} \times\TCBaire))$ we define $S_p:=\{ \sigma \st \coding{p[\abslength{\sigma}],\sigma}\in S \}$.
	We define the forward Weihrauch functional as the map $\Phi:=\coding{p_1,p_2} \mapsto ([S_p], (p_1,0^\omega))$. Notice that, since $F$ is total, $\Phi(\coding{p_1,p_2})$ is a correct input for $\sTCBaire\times F$. Let $(\coding{b}\concat \coding{y,z,w},r)\in (\sTCBaire\times F)([S_p], (p_1,0^\omega))$. We claim that a solution for $\mflim^{(n)}\,\Phi_e (p_1,\TCBaire(p_2))$ is $y$ if $b=1$ or is $r$ if $b=0$.
    
    Assume that $b=1$, i.e.\ that $\coding{y,z,w}\in [S_p]$. Then $D(\coding{p_1,p_2},y,z)$ holds, i.e.\ $y=\mflim^{(n)}(\Phi_e(p_1,z))$. Therefore it is enough to show that $z\in\TCBaire([p_2])$. Assume that $[p_2]\neq\emptyset$ (the other case is trivial). Since $\coding{y,z,w}\in [S_p]$, we have that $P(\coding{p_1,p_2},y,z)$ holds and therefore $z\in [p_2]$.
    
    Assume now that $b=0$, i.e.\ for all $y,z$ there is no $w$ s.t.\ $\coding{y,z,w}\in[S_p]$. If $[p_2]\neq\emptyset$ then choose $z\in [p_2]$ and let $y=\mflim^{(n)}\Phi_e(p_1,z)$. We then have that $D(\coding{p_1,p_2},y,z)$ and $P(\coding{p_1,p_2},y,z)$ hold. Therefore there exists $w$ s.t.\ $\coding{y,z,w}\in[S_p]$, which is a contradiction. This implies that $[p_2]=\emptyset$ and therefore $0^\omega \in\TCBaire([p_2])$ and $(p_1,0^\omega)\in\dom(\mflim^{(n)}\Phi_e)$. Therefore $r=F(p_1,0^\omega)\in \mflim^{(n)}\Phi_e(p_1,\TCBaire([p_2]))$ and this concludes the proof.
\end{proof}
We are now ready to prove the following characterization of arithmetic reducibility to $\TCBaire$, conjectured by Arno Pauly during the BIRS-CMO 2019 workshop ``Reverse Mathematics of Combinatorial Principles".
\begin{theorem}
	\thlabel{thm:arith_reduction_tcbaire}
	For every multivalued function $f$, 
	\[ f\weiarithreducible \TCBaire \iff (\exists n)(f\weireducible \sTCBaire\times \mflim^{(n)}). \]
\end{theorem}
\begin{proof}
	The right-to-left implication follows from \thref{prop:ccantor_arith_equiv_id} and \thref{prop:compprod_arith_reduction}, as $\sTCBaire \weireducible \LPO\compproduct\TCBaire$.

	To prove the left-to-right implication, assume that there exists $m$ s.t.\ $f\weireducible \mflim^{(m)}\compproduct\TCBaire\compproduct\mflim^{(m)}$. Notice that for every single-valued $k$ and every $g,h$ we have 
	\begin{equation}\tag{$\star$}
		(g\times h)\compproduct k \weireducible (g\compproduct k) \times (h\compproduct k).
	\end{equation}
	This fails for multivalued $k$, as shown in \cite[Prop.\ 4.9(19)]{BP16}. We can therefore write
	\begin{align*}
		\mflim^{(m)}\compproduct\TCBaire\compproduct\mflim^{(m)} & \weireducible (\sTCBaire\times \mflim^{[3m+6]})\compproduct\mflim^{[m+1]} &  \text{\thref{thm:lim^n*tcbaire_<_lpo*tcbaire_x_lim^k}} \\
			& \weireducible (\sTCBaire\compproduct\mflim^{[m+1]}) \times \mflim^{[4m+7]} & (\star)\\
			& \weireducible (\LPO\compproduct\TCBaire\compproduct\mflim^{[m+1]}) \times \mflim^{[4m+7]} & \\
			& \weiequiv (\LPO\compproduct(\TCBaire \times \mflim^{[m+1]}))\times \mflim^{[4m+7]} & \text{\thref{thm:tcbaire*lim=tcbaire_x_lim}} \\
			& \weireducible (\LPO\compproduct\mflim^{[m+1]}\compproduct\TCBaire) \times \mflim^{[4m+7]} & \\
			& \weireducible (\mflim^{[m+2]}\compproduct\TCBaire) \times \mflim^{[4m+7]} &  \\
			& \weireducible \sTCBaire \times\mflim^{[3m+9]}\times \mflim^{[4m+7]}  &  \text{\thref{thm:lim^n*tcbaire_<_lpo*tcbaire_x_lim^k}} \\
			& \weiequiv \sTCBaire \times\mflim^{[n]}, &
	\end{align*}
	where $n=\max\{3m+9, 4m+7\}$.
\end{proof}

\begin{corollary}
	\thlabel{thm:tcbaire_arit_<_openRT}
	$\TCBaire\strictlyweiarithreducible \openRamsey$.
\end{corollary}
\begin{proof}
	The reduction follows from $\TCBaire\weireducible \CCantor\compproduct \openRamsey$ (\thref{prop:tcbaire<=ccantor_comp_openRamsey}) using \thref{prop:ccantor_arith_equiv_id} and \thref{prop:compprod_arith_reduction}. The fact that the reduction is strict follows from \thref{thm:arith_reduction_tcbaire} as $\openRamsey\not\weireducible \sTCBaire \times \mflim^{(k)}$ for any $k$ (\thref{thm:openRT_not_<=_sTCbaire_x_lim^n}).
\end{proof}

\begin{theorem}
	\thlabel{thm:openRamsey_arith<_findHopen}
	$\openRamsey\strictlyweiarithreducible \findHopen$.
\end{theorem}
\begin{proof}
It suffices to show that $\openRamsey\strictlyweiarithreducible \CBaire\times \openRamsey$, as $\CBaire\times \openRamsey \weireducible \findHopen$ follows from $\openRamsey\weireducible \findHopen$ (see \thref{thm:lim^n*openRT<=findHopen}), $\CBaire\weireducible \findHopen$ (see \thref{prop:tcbaire<=findHopen}) and the fact that $\findHopen$ is closed under parallel product (\thref{prop:findHopen_closed_product}).

The reduction is trivial, so we only need to prove the separation. Notice
that the analogue of \thref{prop:product_with_NHA_general} for arithmetic
Weihrauch reduction holds (the same proof works by replacing ``computable"
with ``arithmetic'' and $\weireducible$ with $\weiarithreducible$). This allows us to repeat the proof of \thref{prop:product_with_NHA}, obtaining that $\CBaire \times \openRamsey \weiarithreducible \openRamsey$ implies $\openRamsey\weiarithreducible\CBaire$. This contradicts \thref{thm:cbaire_arith<_tcbaire} and \thref{thm:tcbaire_arit_<_openRT}.
\end{proof}

\begin{theorem}
	$\openRamsey^* \weiarithequiv \TCBaire^*$.
\end{theorem}
\begin{proof}
	The right-to-left reduction is a trivial consequence of $\TCBaire\strictlyweiarithreducible \openRamsey$ (\thref{thm:tcbaire_arit_<_openRT}).

	To prove the left-to-right reduction we first notice that
	\[ \openRamsey \weireducible \sTCBaire\times \TwfindHopen. \]
	Indeed, given an open set $P$ we can consider the input $([\closedsidesoltree{\coding{P}}],P)$ for $\sTCBaire\times \TwfindHopen$. This is clearly a valid input as both functions are total. Let $(\coding{b}\concat x,f)\in (\sTCBaire\times \TwfindHopen)([\closedsidesoltree{\coding{P}}],P)$. If $b=1$ then $x\in \HomSol(P)\setdifference{P}$ (\thref{lem:solution_closed_side_iff_path_through_tree}), and therefore $x\in\openRamsey(P)$. If $b=0$ then $[\closedsidesoltree{\coding{P}}]=\emptyset$, which implies that $P\in\dom(\wfindHopen)$ and hence $f\in\wfindHopen(P)$.

	We then have
	\[ \openRamsey \weireducible \sTCBaire\times \TwfindHopen \weireducible \sTCBaire\times \CBaire \weiarithreducible \TCBaire\times \TCBaire, \]
	where $\TwfindHopen \weireducible \CBaire$ follows from $\TwfindHopen\weireducible \ATR_2$ (\thref{thm:twfindopen<=atr2}). From this $\openRamsey^*\weiarithreducible\TCBaire^*$ follows immediately.
\end{proof}

\section{Conclusions}
\label{sec:conclusions}
Some problems resisted full characterization. In particular two questions remain open:

\begin{question}
	$\wfindHclosed\weiequiv \CBaire$?
\end{question}
\begin{question}
	$\CBaire\weireducible\openRamsey$?
\end{question}

Observe that a positive answer to the first question automatically yields a positive answer to the second one by \thref{cor:wfindHclosed<openRT}. We can expect that answering one of the two questions can shed light on the other.

As already observed in \cite{KMP20}, there is not a single ``analogue'' of $\mathrm{ATR}_0$ in the context of Weihrauch reducibility, and theorems that are equivalent from the reverse mathematics point of view can exhibit very different behaviors when phrased as multivalued functions. 

Notice in particular that the classical proofs of the equivalences, over $\mathrm{RCA}_0$, of $\mathrm{ATR}_0$ and the open and clopen Ramsey theorems (\cite{Simpson09}) are useful only to establish that $\UCBaire\weiequiv \wfindHopen \weiequiv \wfindHclopen$.

Finding a homogeneous solution that lands in an open set, when there are also solutions that avoid it, is a much harder problem. In particular, notice that a natural candidate for $\boldfacePi^1_1\mathrm{-CA}_0$ in the Weihrauch lattice is $\parallelization{\chiPi}$. The fact that $\findHopen$ computes $\chiPi$ and is closed under parallel product implies that $\chiPi^*\weireducible \findHopen$. This naturally leads to the following question:
\begin{question}
	$\parallelization{\findHopen}\weireducible\findHopen$?
\end{question}
A positive answer to this question would locate $\findHopen$ in the realm of $\boldfacePi^1_1\mathrm{-CA}_0$, in sharp contrast with what happens in reverse mathematics.

Let us now derive a few computability-theoretic corollaries.

\begin{corollary}[{Solovay's theorem \cite[Thm.\ 1.8]{Solovay78}}]
	\thlabel{thm:solovay} If $P\subset\Ramsey$ is open with computable code then either there is an hyperarithmetic homogeneous solution
landing in $P$ or there is a homogeneous solution avoiding $P$.
\end{corollary}
\begin{proof}
This follows from \thref{thm:ucbaire_hyperarithmetic_solution} and the fact that $\wfindHopen\weiequiv \UCBaire$ (\thref{thm:open<ucbaire}).
\end{proof}

The following result is attributed to Solovay \cite{Solovay78} (see \cite[Thm.\ 1]{Mansfield78} for an explicit proof).
\begin{corollary}
	The set of homogeneous solutions for a clopen set with computable code always contains an hyperarithmetic element.
\end{corollary}
\begin{proof}
	This follows from \thref{thm:ucbaire_hyperarithmetic_solution} and the fact that $\clopenRamsey\weiequiv\UCBaire$ (\thref{thm:ucbaire=clopenramsey}).
\end{proof}

\begin{corollary}
	There is a clopen set $D\subset\Ramsey$ with computable code s.t.\ every homogeneous solution that lands in $D$ is not hyperarithmetic.
\end{corollary}
\begin{proof}
	This follows from the fact that $\findHclopen\weiequiv\CBaire$ (\thref{thm:findclopen=findclosed=cbaire}): if every computable clopen set had an hyperarithmetic solution landing in itself then every computable instance of $\CBaire$ would have a hyperarithmetic solution, contradicting $\NHA\weireducible\CBaire$ (\cite[Cor.\ 3.6]{KMP20}).
\end{proof}
	
\begin{corollary}
	\thlabel{cor:complexity_of_solutions}
	Every open set $P\subset\Ramsey$ with computable code has a homogeneous solution $f$ that is strictly Turing reducible to Kleene's $\KleeneO$.
\end{corollary}
\begin{proof}
	It follows from the proof of Gandy basis theorem (see \cite[Chap.\ III, Thm.\ 1.4]{SacksHRT}) that $\{f\st f\strictlyturingreducible \KleeneO\}$ is a basis for the $\Sigma^1_1$ predicates. If $P\in\dom(\wfindHopen)$ then, by \thref{thm:solovay}, it has an hyperarithmetic solution. Otherwise $P\in\dom(\findHclosed)$ hence, by \thref{lem:solution_closed_side_iff_path_through_tree}, a homogeneous solution for $P$ can be computed from any element of $[\closedsidesoltree{\coding{P}}]$ (the tree $\closedsidesoltree{\coding{P}}$ is computable from $\coding{P}$, see \thref{lem:computable_operations}). By the Gandy basis theorem the claim follows.
\end{proof}
In particular \thref{cor:complexity_of_solutions} shows that the difference, in the (arithmetic) Weihrauch lattice, between $\openRamsey$ and $\CBaire$ cannot be explained in terms of complexity of the solutions but rests entirely on the lack of uniformity.

\let\oldaddcontentsline\addcontentsline% Store \addcontentsline
\renewcommand{\addcontentsline}[3]{}% Make \addcontentsline a no-op
\bibliographystyle{mbibstyle}
\bibliography{bibliography}
\let\addcontentsline\oldaddcontentsline

\end{document}